\documentclass[11pt,reqno]{amsart}%
\usepackage{amsmath}
\usepackage{graphicx}%
\usepackage{amsfonts}%
\usepackage{amssymb}
\usepackage{amsthm}
\usepackage{hyperref}
\usepackage{color}
\definecolor{bleuf}{rgb}{0.1,0.1,0.6}
\hypersetup{
       colorlinks = true,           
       breaklinks = true,          
       urlcolor = bleuf,           
       linkcolor =bleuf,                  
       citecolor = bleuf,                     
       linktocpage = true,
       pdftitle={Markovian loop clusters on graphs},   
       pdfauthor={Y. Le Jan S. Lemaire},                  
    }

\newtheorem{theorem}{Theorem}[section]
\newtheorem{lem}[theorem]{Lemma}
\newtheorem{proposition}[theorem]{Proposition}

\newtheorem*{bibli}{Theorem}

\theoremstyle{definition}

\newtheorem{example}[theorem]{Example}
\newtheorem{remark}[theorem]{Remark}

\numberwithin{equation}{section}
\numberwithin{theorem}{section}

\DeclareMathOperator{\NN}{\mathbb{N}}
\DeclareMathOperator{\ZZ}{\mathbb{Z}}
\DeclareMathOperator{\RR}{\mathbb{R}}
\DeclareMathOperator{\Pd}{\mathbb{P}}
\DeclareMathOperator{\Ed}{\mathbb{E}}
\DeclareMathOperator{\un}{1\hspace{-.29em}I}
\renewcommand{\L}{\mathcal{L}}
\DeclareMathOperator{\diag}{diag}
\DeclareMathOperator{\card}{card}
\DeclareMathOperator{\leb}{Leb}
\newcommand\Psub[2]{\mathcal{P}_{#2}(#1)}
\newcommand\enum[2]{\{#1,\ldots,#2\}}
\newcommand\enu[1]{\{1,\ldots,#1\}}
\newcommand\ld{\ell}
\newcommand\La{\mathcal{DL}_{\alpha}}
\newcommand\Ca{\mathcal{C}_{\alpha}}
\newcommand\DL{\mathcal{DL}}
\newcommand\DLp{\mathcal{D}\dot{\mathcal{L}}}
\newcommand\lp{\dot{\ell}}
\newcommand\mup{\dot{\mu}}
\DeclareMathOperator{\Per}{Per}
\DeclareMathOperator{\LL}{L}
\newcommand\inter{\;\text{intersects}\;}
\newcommand\ninter{\;\text{does not intersect}\;}
\renewcommand{\epsilon}{\varepsilon}
\graphicspath{{Figures/}}

\begin{document}
\title[Markovian loop clusters]{Markovian loop clusters on graphs}
\author{Yves Le Jan}
\date{April, 2013}
\address{Universit\'e Paris-Sud, Laboratoire de Math\'ematiques, UMR 8628, Orsay
{F-91405}; CNRS, Orsay, {F-91405}; IUF.}
\email{yves.lejan@math.u-psud.fr}
\author{Sophie Lemaire}
\address{Universit\'e Paris-Sud, Laboratoire de Math\'ematiques, UMR 8628, Orsay
{F-91405}; CNRS, Orsay, {F-91405}.}
\email{sophie.lemaire@math.u-psud.fr}

\begin{abstract}
We study the loop clusters induced by Poissonian ensembles of Markov loops on a finite or countable graph (Markov loops can be viewed as excursions of Markov chains with a random starting point, up to re-rooting). Poissonian ensembles are seen as a Poisson point process of loops indexed by `time'. The evolution in time of the loop clusters defines a coalescent process on the vertices of the graph.  After a description of some general properties of the coalescent process, we address several aspects of the loop clusters  defined by a simple random walk killed at a constant rate on three different graphs:  the integer number line $\mathbb{Z}$, the integer lattice $\mathbb{Z}^d$ with $d\geq 2$ and the complete graph. These examples show the relations  between Poissonian ensembles of Markov loops and other models: renewal process, percolation and random graphs.
\end{abstract}
\maketitle
\section*{Introduction}
The notion of Poissonian ensembles of Markov loops (loop soups) was introduced
by Lawler and Werner in \cite{lw}  in the context of two
dimensional  Brownian motion (it already appeared informally 
in \cite{sym}): the loops of a Brownian loop soup on a domain $D\subset \mathbb{C}$  are the points of a Poisson point process with intensity $\alpha \mu$ where $\alpha$ is a positive real and $\mu$ is the Brownian loop  measure on $D$. Loop clusters induced by a Brownian loop soup were used to give a construction of the conformal loop ensembles (CLE) in \cite{WernerCRAS03} and \cite{shw}. They are defined as follows: \emph{two loops $\ell$ and $\ell'$ are said to be in the same cluster if one can find a finite chain of loops $\ell_0,\ldots,\ell_n$ such that $\ell_0=\ell$, $\ell_n=\ell'$ and for all $i\in \enu{n}$, $\ell_{i-1}\cap \ell_i\neq \emptyset$}.  
Poissonian ensembles of Markov loops can also be defined on  graphs. They were studied in details in  \cite{stfl} and \cite{lawlerlimic}. In particular  Poissonian ensembles of loops on the integer lattice $\ZZ^2$ induced by  simple (nearest neighbor) random walks give a discrete version of Brownian loop soup (see \cite{LawlerFerreras}). 

The aim of this paper is to study  loop clusters  induced by Markov loop ensembles on graphs.  The facts
presented here, which are built on the results presented in \cite{stfl}, are more elementary than in the Brownian loop soup theory, but they point out  that Poissonian ensembles of Markov loops and related partitions are of more
general interest. The examples we develop show relations with several
theories: coalescence, percolation and random graphs.

In Section 1, we recall the different objects needed to define the Markovian loop clusters on a graph and the coalescent process induced by the partition of vertices they define. In Section 2, we state general properties of the clusters and establish some formulae useful to study the semigroup of the coalescent process. In the last three sections we address several aspects of the loop clusters induced by a simple random walk killed at a constant rate $\kappa$ on different graphs: 
\begin{itemize}
\item On  $\ZZ$, the loop cluster model reduces to a renewal process. We establish a convergence result by rescaling space by $\sqrt{\epsilon}$,   killing rate by $\epsilon$ and let $\epsilon$ converge to $0$. 
\item The loop cluster model can be seen as a percolation model with two parameters $\alpha$ and $\kappa$. Bernoulli percolation appears as a limit if $\alpha$ and $\kappa$ tend to $+\infty$ so that $\frac{\alpha}{\kappa^2}$ converges to a positive real.  We show that a non-trivial percolation threshold  occurs for the loop cluster model on  the  lattice $\ZZ^{d}$ with $d\geq 2$. 
\item As a last example, we consider the complete graph. We give a simple construction of the coalescent process on the complete graph and we deduce a  construction of a coalescent process on the interval $[0,  1]$. Letting the size of the graph increase to $+\infty$, we determine the asymptotic distribution of the coalescence time: it appears to be essentially the same as  in the Erd\"os-R\'enyi random graph model (see \cite{ErdosRenyi59}) though the lack of independence makes the proof significantly  more difficult. 
\end{itemize}
\section{Setting}
We consider a finite or countable simple graph $\mathcal{G}=(X,\mathcal{E})$ endowed with positive
conductances $C_{e},\ e\in \mathcal{E}$ and a positive measure $\kappa=\{\kappa_{x},\ x\in X\}$ (called the \emph{killing measure}).  
 We  denote by $C_{x,y}$ the
conductance of the edge between vertices $x$ and $y$ and set $\lambda_{x}=\sum_{y\in X, \{x,y\}\in\mathcal{E}}C_{x,y}+\kappa_{x}$   for every $x\in X$. The conductances and the killing measure $\kappa$  induce a sub-stochastic matrix $P$: $P_{x,y}=\frac{C_{x,y}}{\lambda_x}\un_{\{\{x,y\}\in\mathcal{E}\}}$, $\forall x,y\in X$.  
We assume that $P$ is irreducible. 
\subsection{Loop measure}
A discrete based loop $\ell$ of length $n\in\NN^*$ on $\mathcal{G}$ is defined as an element of $X^n$; it can be extended to an infinite periodic sequence.  
Two based loops of length $n$  are said equivalent if they only differ by a shift of their coefficients ({\it i.e.} the based loops of length $n$, $(x_1, \ldots, x_n)$ is equivalent to the based loop of length $n$ $(x_i,\ldots,x_{n},x_1,\ldots,x_{i-1})$ for every $i\in\{2,\ldots,n\}$). A discrete loop is an equivalent class of based loops for this equivalent relation (an example is drawn in Figure \ref{figloop}).   Let $\DL(X)$ (resp. $\DLp(X)$) denote  the set of discrete  loops (resp. discrete based loops) of length at least 2 on $X$.  We associate to each based loop $\ell=(x_1,\ldots,x_n)$ of length $n\geq 2$ the weight $\mup(\ell)=\frac{1}{n}P_{x_1,x_2}P_{x_2,x_3}\ldots P_{x_n,x_1}$. This defines a measure $\mup$ on $\mathcal{DL}(X)$ which is invariant by the shift and therefore  induces a measure $\mu$ on $\DL(X)$.

\begin{figure}[hbt]
\fbox{%
\begin{minipage}{0.4\textwidth}
\centerline{\includegraphics[scale=0.45]{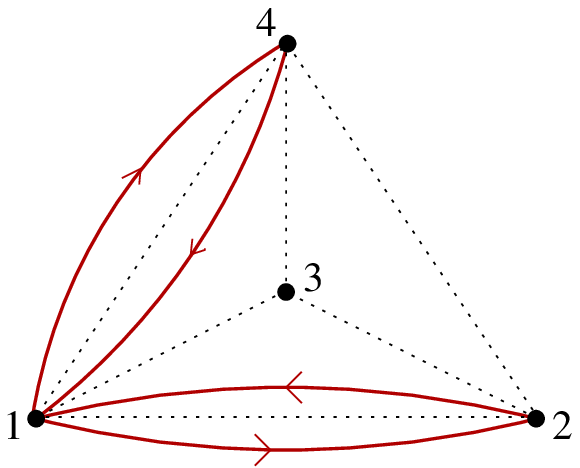}}
\end{minipage}
 \begin{minipage}{0.5\textwidth}
{\small  Representation of a loop $\ell$ of length 4 on the complete graph $K_4$ corresponding to the equivalence class of the based loop ${\lp=(4,1,2,1)}$; the measure of this loop is
\[\mu(\ell)=P_{4,1}P_{1,2}P_{2,1}P_{1,4}.\] }
\end{minipage}%
}
\caption{\label{figloop}}%
\end{figure}

\begin{remark}[\emph{Doob's $h$-transform}]
If $h$ is a  positive function on $X$ such that $(P-I)h\leq 0$,  a new set of conductances $C^{\{h\}}$ and a new killing measure $\kappa^{\{h\}}$ can be defined as follows: 
$C^{\{h\}}_{x,y}=h(x)h(y)C_{x,y}$ for every $\{x,y\}\in \mathcal{E}$ and $\kappa^{\{h\}}_x=h(x)[(I-P)h](x)\lambda_x$ for every $x\in X$. This modification corresponds to the Doob's $h$-transform: the associated transition matrix $P^{\{h\}}$ verifies $P^{\{h\}}_{x,y}=\frac{h(y)}{h(x)}P_{x,y}$, $\forall x,y\in X$. It is a self-adjoint operator on $\LL^2(h^2\lambda)$ since ${P^{\{h\}}=T_h^{-1}PT_{h}}$ with $T_h:\begin{array}{lcl}\LL^2(h^2\lambda)&\rightarrow& \LL^2(\lambda)\\
f&\mapsto& hf       \end{array}.$
The loop measure $\mu$ is invariant 
under Doob's $h$-transform. \\
The loop measure can be defined without assuming that $P$ is a substochastic matrix. A  weaker condition would be to assume that $P$ is a positive and contractive matrix. Taking such matrices does not extend the set of loop measures.  Indeed,  Perron-Frobenius theorem states that if $Q$ is a positive and irreducible matrix and if its  spectral radius $\rho(Q)$ is smaller or equal to 1, then there exists a positive function $h$ on $X$ such that $Ph=\rho(Q)h$.  The $h$-transform of $Q$ is then a substochastic matrix.   
\end{remark}
\subsection{Poisson loop sets}
Let $\mathcal{DP}$ be a Poisson point process with intensity $\leb\otimes
\mu$ defined on $\RR_{+}\otimes\DL(X)$. For $\alpha\geq 0$, let $\DL_{\alpha}$ denote the projection of the set ${\mathcal{DP}\cap([0,  \alpha]\times\DL(X))}$
on $\DL(X)$;  $(\DL_{\alpha})_{\alpha\geq 0}$ is an increasing family of loop
sets which has stationary independent increments.
It coincides with the family  of non-trivial discrete loop sets induced by the Poisson Point process of continuous-time loops defined in  \cite{stfl}. 
\subsection{Coalescent process}
An edge $e\in\mathcal{E}$ is said to be \emph{open at time $\alpha$} if $e$ is traversed by at least one loop of $\mathcal{DL}_{\alpha}$. The set of open edges  defines a subgraph $\mathcal{G}_{\alpha}$ with vertices $X$. The connected components of $\mathcal{G}_{\alpha}$  define a partition of $X$ denoted by $\mathcal{C}_{\alpha}$. The
elements of the partition $\mathcal{C}_{\alpha}$ are the loop clusters defined
by $\mathcal{DL}_{\alpha}$ (an example is drawn in Figure \ref{figcluster}). This paper is devoted to the study of $\mathcal{C}_{\alpha}$. \\
\begin{figure}[ht]
\fbox{%
 \begin{minipage}{.25\textwidth}
\begin{center} \includegraphics[scale=0.38]{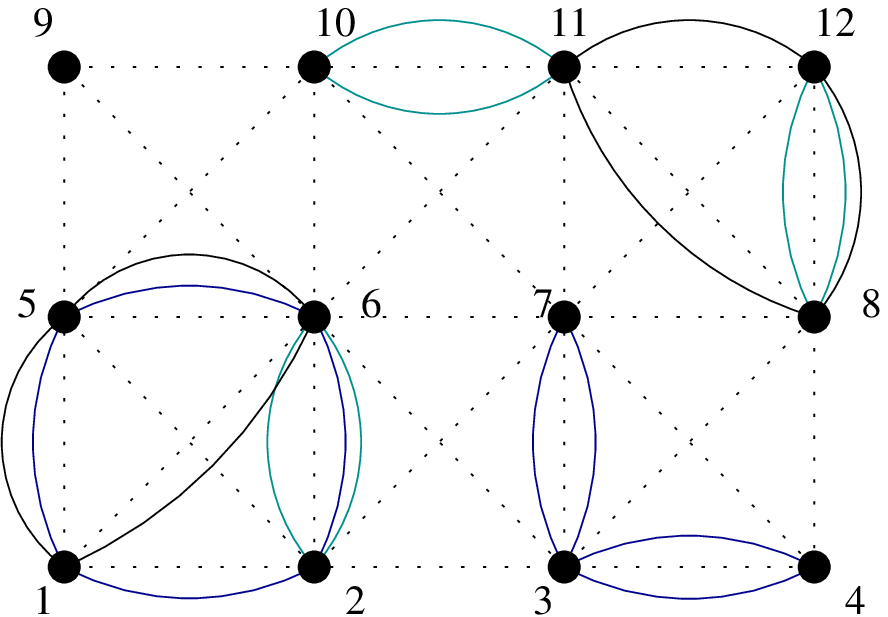} \\
 {\small $\mathcal{DL}_{\alpha}$} 
\end{center}
\end{minipage}$\rightarrow$
 \begin{minipage}{.25\textwidth}
 \begin{center}  \includegraphics[scale=0.38]{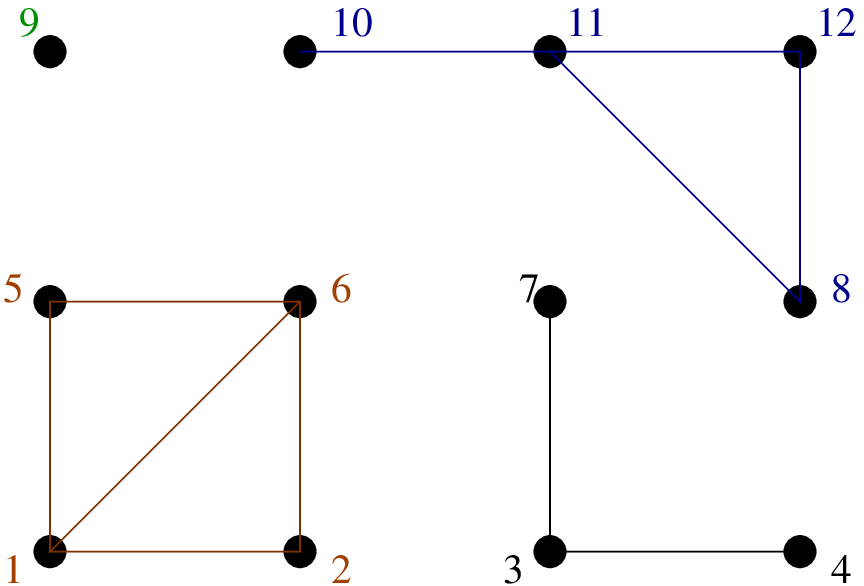} \\
 {\small $\mathcal{G}_{\alpha}$} 
\end{center}
\end{minipage}
$\rightarrow$\quad
\begin{minipage}{.32\textwidth}
 $\mathcal{C}_{\alpha}$ is  a partition with 4 blocks: $\{1,2,5,6\}, \{3,4,7\},$ \\ $\{8,10,11,12\},\{9\}$.
\end{minipage}}
\caption{\label{figcluster}Example of a loop set $\mathcal{DL}_{\alpha}$ at time $\alpha$ on a finite graph $\mathcal{G}$, the subgraph $\mathcal{G}_{\alpha}$ defined by the open edges and the partition $\mathcal{C}_{\alpha}$ induced by its loop clusters (the dotted-lines on the left figure represent the  edges of $\mathcal{G}$; $\mathcal{DL}_{\alpha}$ consists  of three loops of length 2, two loops of length 3 and two loops of length~4.) }
\end{figure}
\begin{remark}
 If $A$ is a subset of $X$, let us define the $\DL_{\alpha}$-neighborhood of $A$: $$\tau_{\alpha}(A)=A\cup\{x\in X,\; \exists \ell\in \DL_{\alpha} \text{ visiting } A \text{ and } x\}.$$

Given any $(x,y)\in X^2$, set $x\underset{\alpha}{\sim} y \text{ if and only if } \exists k\in \NN^*\; \text{ such that } y\in\tau^{k}_{\alpha}(\{x\})$ (in the example drawn figure \ref{figcluster}, $\tau_{\alpha}(\{10\})=\{10,11\}$ and for every $k\geq 2$, $\tau^{k}_{\alpha}(\{10\})=\tau_{\alpha}(\{10,11\})=\{8,10,11,12\}$ for instance). 
This defines  an equivalence relation and the associated partition is $\mathcal{C}_{\alpha}$. 
\end{remark}
\section{General properties of discrete loop clusters}
\subsection{The distribution of the set of primitive discrete loops}
Positive integer powers of a based discrete loops is defined by concatenation: if ${\ell=(x_1,\ldots,x_n)}$ is a based discrete loop of length $n$ then $[\ell]^1=\ell$, $[\ell]^{2}$ is the discrete based loop $(x_1,\ldots,x_n,x_1,\ldots,x_n)$ and so on. As the $m$-th power of equivalent based loops are also equivalent, powers of discrete loops are well-defined. 
A discrete loop $\xi$ is said to be \emph{primitive} if  there is no integer $m\geq 2$ and no discrete based loop $\ell$ such that $\xi$ is the equivalent class of $[\ell]^{m}$; 
Any  discrete loop $\xi\in\mathcal{DL}(X)$ can be represented 
as a power of a unique primitive discrete loop 
 denoted by $\pi\xi$. Let $\mathcal{PL}(X)$ denote the set of 
 primitive discrete loops $X$ of length at least 2 and let $\mathcal{PL}_{\alpha}$ denote the set of primitive discrete
loops defined by $\mathcal{DL}_{\alpha}$. Clearly $\mathcal{C}_{\alpha}$
depends only on $\mathcal{PL}_{\alpha}$. 
\begin{proposition}\label{primitivedistrib}
 The  probability distribution of $\mathcal{P}\L_{\alpha}$ is a product measure $\nu$ on ${\{0,1\}^{\mathcal{PL}(X)}}$ 
defined by \[\nu(\omega,\ \omega_{\eta}=1)=1-(1-\mu(\eta))^{\alpha}.\]
\end{proposition}
\begin{proof}
If $\eta_{1}, \ldots, \eta_{r}$ are  distinct primitive discrete loops, the  sets 
\[L_i=\{\ell\in \mathcal{DL}(X),~\pi\ld=\eta_i\},\; i=1,\ldots, r\] are disjoint, hence the $r$ events 
$E_i=\{\exists \ell\in \mathcal{DL}_{\alpha}\;\text{such that}\;\pi \ld=\eta_i\}$ for ${i\in\{1, \ldots, r\}}$ 
are independent. Therefore, the law of $\mathcal{PL}_{\alpha}$ is a product measure on $\{0,1\}^{\mathcal{PL}(X)}$ 
and for every $\eta\in \mathcal{PL}(X)$, \[\Pd(\eta\in \mathcal{PL}_\alpha)=\Pd(\exists \ell\in
\mathcal{DL}_{\alpha},\;\pi \ld=\eta)=1-\exp(-\alpha\sum_{m=1}^{+\infty}\mu(\ell, \ld=[\eta]^m)). \]
As $\mu([\eta]^m)=\frac{1}{m}\mu(\eta)^m$ for every $m\in\NN^*$, we deduce that\[\Pd(\eta\in \mathcal{PL}_\alpha)=1-\exp(\alpha\log(1-\mu(\eta))).\] 
\end{proof}

It follows from Proposition \ref{primitivedistrib} that the Harris inequality (and also the B-K inequality) holds on
increasing events (see e.g. \cite{WernerBook}). In particular let us say that \emph{a
subset $A$ of $X$ is connected at time $\alpha$} if it is contained in a cluster of
$\mathcal{C}_{\alpha}$. Then by the Harris inequality, 
\begin{multline*}
\Pd(A\;\text{and}\;B\;\text{are both connected at time}\;\alpha)\\\geq
\Pd(A\;\text{is connected at time}\;\alpha)\Pd(B\;\text{is connected at time}\;\alpha). 
\end{multline*}
In a similar way, we say that \emph{a subgraph is open at time $\alpha$} if all its edges are
traversed by jumps of loops in $\mathcal{DL}_{\alpha}$. Then the same inequality holds.\\
These inequalities can be extended to any number of increasing events. In
particular, 
\begin{itemize}
 \item For every partition $\pi=(B_i)_{i\in I}$ of $X$ into non-empty blocks,  the probability that  $\mathcal{C}_{\alpha}$ is a coarser\footnote{Given two partitions $\sigma$ and $\pi$ of $X$, $\sigma$ is said to be \emph{coarser} than $\pi$ (denoted by $\sigma \succeq \pi$) or $\pi$ is said to be \emph{finer} than $\sigma$ (denoted by  $\pi \preceq \sigma$) if every block of $\pi$ is a subset of a block of $\sigma$. } partition than $\pi$ satisfies:
\[\mathbb{P}(\mathcal{C}_{\alpha}\succeq\pi)\geq\prod_{i\in I}\mathbb{P}(B_{i}\;\text{is connected at time}\;\alpha).\]
\item If $\mathcal{F}\subset \mathcal{E}$, then $\Pd(\mathcal{F} \text{ is open at time }\alpha)\geq \underset{e\in\mathcal{F}}{\prod}\Pd(e\; \text{is open at time}\;\alpha)$. 
\end{itemize}

\subsection{The transition rate of the coalescent process}
 The evolution in $\alpha$ of $\mathcal{C}_{\alpha}$ defines a partition-valued Markov chain. 
Let $\pi$ be a partition of $X$ into non-empty blocks $\{B_i,\ i\in I\}$. From state $\pi$, the only possible transitions are to a partition $\pi^{\oplus J}$ obtained by merging  blocks indexed by some subset $J$ of $I$ to form one block $B_J=\cup_{j\in J}B_j$ and leaving all other blocks unchanged.  The transition rate from  $\pi$ to   $\pi^{\oplus J}$ is
\begin{align}
\label{eqtransrate}
\tau_{\pi,\pi^{\oplus J}}=&\mu(\ell,\; \forall j\in J,\; \ell \inter B_j\;\text{and}\; \forall u\not\in J,\; \ell \ninter B_u)\\
=&\sum_{k\geq 2}\frac{1}{k}\sum_{(x_1,\ldots,x_k)\in W_k(B_j,\;j\in J)}P_{x_1,x_2}\ldots P_{x_{k},x_{1}}\notag
\end{align}
where $W_k(B_j,\; j\in J)$ is the set of $k$-tuples of elements of $B_J$ having at least one element that belongs to each block $B_j,\; j\in J$: \[W_k(B_j,\;j\in J)=\{(x_1,\ldots,x_k)\in(\cup_{j\in J}B_j)^k,\ \forall j\in J,\; x_u\in B_j \;\text{for some}\;u\}.\]
If the graph $\mathcal{G}$ is finite, these transition rates can be expressed with  determinants of Green's functions of subgraphs.  
In order to describe the formula, let us introduce some notations; Let $G$ denote the Green's function of $\mathcal{G}$: $G=(\lambda I-C)^{-1}$ where $\lambda=(\lambda_x)_{x\in X}$, $I$ is the  identity matrix and $C$ is the conductance matrix.  For every subset $F$ of vertices, set $F^c=X\setminus F$ and let $G^{(F)}$ denote the Green's function of the subgraph $(F,\mathcal{E}_{|F\times F})$ of $\mathcal{G}$: $ G^{(F)}$ is the inverse of the matrix  $(\lambda I-C)_{|{F\times F}}$. 
\begin{proposition}
Let us assume that  $\mathcal{G}=(X,\mathcal{E})$ is a finite graph. \\
For every partition $\pi=\{B_i,\; i\in I\}$ of $X$ and every subset $J$ of $I$ with at least two elements: 
\begin{equation}
\label{eqtransratecrible}
\tau_{\pi,\pi^{\oplus J}}=\sum_{I\subsetneqq J}(-1)^{|I|}\log(\det(G^{(\cup_{u\in J\setminus I}B_u)})).
\end{equation}
\end{proposition}
\begin{proof}
By the inclusion-exclusion principle, 
 \begin{align*}\tau_{\pi,\pi^{\oplus J}}= & \mu\Big(\ell,\;  \ell \ninter \underset{j\not\in J}{\cup}B_j\Big)\\&\qquad-\mu\Big(\ell,\; \exists u\in J,\; \ell\ninter \underset{j\in J^c\cup\{u\}}{\cup}B_j\Big)\\=&\sum_{K\subsetneqq  J}(-1)^{|K|}\mu(\ell,\;  \ell \ninter \underset{u\in J^c\cup K}{\cup}B_u).
 \end{align*}
To conclude, we  use that:
\begin{itemize}
\item For a subset  $F$ of vertices,\[\mu(\ell,\; \ell \ninter F)=\mu(\DL(F^c))=\log\Big(\det(G^{(F^c)})\prod_{x\in F^c}\lambda_x\Big).\]
\item For every family of reals $(a_u)_{u\in J}$ indexed by a finite subset $J$ with at least two elements, 
 \begin{equation}\label{sumalt}
\sum_{I\subsetneqq J}(-1)^{|I|}\sum_{u\in J\setminus I}a_u=0.
 \end{equation}
\end{itemize}
\end{proof}
\begin{example}
\label{extransKn}
Let us consider the complete graph $K_n$ with unit conductances and a uniform killing measure with intensity $\kappa$. The  transition matrix $P$ verifies: $P_{x,y}=\frac{1}{n-1+\kappa}\un_{\{x\neq y\}}$ for every  vertices $x$ and $y$. Thus for every subset $F$ of vertices,
${\det(G^{(F)})=((n+\kappa)^{|F|-1}(n+\kappa-|F|))^{-1}}$  (the computation of this determinant is detailed in  Lemma \ref{detmatrix}).  
 Using \eqref{eqtransratecrible} and \eqref{sumalt},  we obtain:  
 \begin{equation}\tau_{\pi,\pi^{\oplus J}}=\sum_{I\subsetneqq J}(-1)^{|I|+1}\log\Big(1-\frac{1}{n+\kappa}\sum_{u\in J\setminus I}|B_u|\Big).
 \end{equation}
\end{example}
\subsection{The semigroup of the coalescent process on a finite graph}
In this section, we assume that the graph $\mathcal{G}$ is finite. For a partition $\pi$ of $X$, let  $\Pd_{\pi}(\cdot)$ denote the conditional probability $\Pd(\cdot \mid \mathcal{C}_{0}=\pi)$.  
The probability that   $\mathcal{C}_{\alpha}$ is finer than a given partition of $X$ has a simple expression: 
\begin{lem}\label{semigrouppartition}
Let us assume that the graph $\mathcal{G}=(X,\mathcal{E})$ is finite. Let $\pi$ be a partition of $X$ into non-empty blocks $\{B_i,\ i\in I\}$. For every partition $\pi_0$ of $X$,
 \begin{equation}\label{semigroupdetG}
\Pd_{\pi_{0}}(\mathcal{C}_{\alpha}\preceq \pi)=\left(\frac{\prod_{i\in I}\det(G^{(B_i)})}{\det(G)}\right)^{\alpha}\un_{\{\pi_0 \preceq\pi\}}. 
\end{equation}
\end{lem}
\begin{proof} Let us assume that $\pi$ is coarser than $\pi_0$. 
The event `\emph{$\mathcal{C}_{\alpha}$ is finer than $\pi$}' means that every  loop of $\La$ is included in a block of $\pi$.  Therefore,  \[\Pd_{\pi_{0}}(\mathcal{C}_{\alpha}\preceq \pi)=\exp\Big(-\alpha\Big(\mu(\DL(X))-\sum_{i\in I}\mu(\DL(B_i))\Big)\Big) \]
since the set of loops in each block $B_i$ of $\pi$ at time $\alpha$  defines independent Poisson point processes. 
To conclude, we use that \[\mu(\DL(X))=\log\Big(\det(G)\prod_{x\in X}\lambda_x\Big)\] and \[\mu(\DL(B_i))=\log\Big(\det(G^{(B_i)})\prod_{x\in B_i}\lambda_x\Big).\]
\end{proof}

An explicit formula for $\Pd_{\pi_0}(\mathcal{C}_{\alpha}=\pi)$ can be derived from Lemma \ref{semigrouppartition}.  Let us first introduce some notations. For  a partition $\pi$,  let $|\pi|$ denote the number of non-empty blocks of $\pi$. For a subset $A$,  let $\pi_{|A}$ denote the restriction of $\pi$ to $A$: $\pi_{|A}$ is a partition of $A$, 
the blocks of which are the intersection of the blocks of $\pi$ with $A$. 
\begin{lem}
Let us consider a finite graph $\mathcal{G}=(X,\mathcal{E})$. 
Let  $\pi_0$ and $\pi$ be two partitions of $X$. If $\pi$ has $k$ non-empty blocks denoted by $B_1\ldots,B_k$ then
\begin{equation}
\label{sgr_alpha}
\Pd_{\pi_0}(\mathcal{C}_{\alpha}=\pi)=\sum_{\tilde{\pi}\preceq \pi}(-1)^{|\tilde{\pi}|-k}\prod_{i=1}^{k}(|\tilde{\pi}_{|B_i}|-1)!\Pd_{\pi_0}(\mathcal{C}_{\alpha}\preceq \tilde{\pi}).
\end{equation}
\end{lem}
\begin{proof}
Let us first assume that  $\pi=\{X\}$. To obtain equation \eqref{sgr_alpha},  it is  sufficient to prove the following identity:
\begin{equation}\label{eqindpartition}
\un_{\{\mathcal{C}_{\alpha}=\{X\}\}}=\sum_{\ell=1}^{|X|}(-1)^{\ell-1}(\ell-1)!\sum_{\pi_{\ell}\in \mathfrak{P}_{\ell}(X)}\un_{\{\mathcal{C}_{\alpha}\preceq \pi_{\ell}\}}.
\end{equation}
where $\mathfrak{P}_{\ell}(X)$ denotes the set of partitions of $X$ with $\ell$ non-empty blocks. 
Let us assume that $\mathcal{C}_{\alpha}$ is a partition with $j$ non-empty blocks. For $\ell\leq j$,  we can construct a partition coarser than $\mathcal{C}_{\alpha}$ with $\ell$ non-empty blocks by choosing how to merge some blocks of $\mathcal{C}_{\alpha}$, that is by choosing a partition of $\enu{j}$ with $\ell$ blocks. Therefore $$\sum_{\pi_{\ell}\in \mathfrak{P}_{\ell}(X)}\un_{\{\mathcal{C}_{\alpha}\preceq \pi_{\ell}\}}= |\mathfrak{P}_{\ell}(\enu{j})|\un_{\{\ell\leq j\}}$$ 
and the right-hand side of \eqref{eqindpartition} is equal to 
${\displaystyle \sum_{\ell=1}^{j}(-1)^{\ell-1}(\ell-1)!|\mathfrak{P}_{\ell}(\enu{j})|}$. By an identity on the Stirling numbers of the second kind (see for example \cite{PitmanLN} equation (1.30) page 22), this sum is equal to 1 if $j=1$ and $0$ if $j\geq 2$.  \\
To prove \eqref{sgr_alpha} for a partition $\pi$ with $k$ non-empty blocks $(B_1,\dots,B_k)$, we consider $\mathcal{C}_{\alpha}$ as a sequence of  partitions on $B_1,\ldots, B_k$ respectively and apply \eqref{eqindpartition}: 
 \[\un_{\{\mathcal{C}_{\alpha}=(B_1,\ldots,B_k)\}}=\prod_{i=1}^{k}\un_{\{\mathcal{C}_{\alpha|B_i}=B_i\}}=\sum_{\ell_1=1}^{|B_1|}\cdots\sum_{\ell_k=1}^{|B_k|}\prod_{i=1}^{k}V_{i,\ell_i}\]
where $V_{i,\ell}=(-1)^{\ell-1}(\ell-1)!\sum_{\tilde{\pi}\in \mathfrak{P}_{\ell}(B_i)}\un_{\{\mathcal{C}_{\alpha|B_i}\preceq \tilde{\pi}\}}$.
Therefore, \begin{alignat*}{2}
\un_{\{\mathcal{C}_{\alpha}=(B_1,\ldots,B_k)\}}=&\sum_{\ell_1=1}^{|B_1|}\!\!\cdots\!\!\sum_{\ell_k=1}^{|B_k|}(-1)^{\ell_1+\ldots+\ell_k-k}\prod_{i=1}^{k}(\ell_i-1)!\hspace{-10px}\sum_{\tilde{\pi}\in \mathfrak{P}_{\ell_1}(B_1)\times \!\cdots \!\times  \mathfrak{P}_{\ell_k}(B_k)}\hspace{-30px}\un_{\{\mathcal{C}_{\alpha}\preceq \tilde{\pi}\}}\\
=&\sum_{\tilde{\pi}\preceq \pi}(-1)^{|\tilde{\pi}|-k}\prod_{i=1}^{k}(|\tilde{\pi}_{|B_i}|-1)!\un_{\{\mathcal{C}_{\alpha}\preceq \tilde{\pi}\}}.
\end{alignat*}
\end{proof}
\begin{example}
Let us consider the complete graph $K_n$ endowed with unit conductances and a uniform killing measure of intensity $\kappa$. If  $\pi$ is a partition of the set of vertices $X$ with $k$ blocks $B_1,\ldots,B_k$ then $$\Pd_{\pi_{0}}(\mathcal{C}_{\alpha}\preceq \pi)=(\frac{\kappa}{\kappa+n})^{\alpha}\prod_{i\in I}(1-\frac{|B_i|}{n+\kappa})^{-\alpha}\un_{\{\pi_0 \preceq\pi\}} $$
and
$$
\Pd_{\pi_0}(\mathcal{C}_{\alpha}=\pi)=\Big(\frac{\kappa}{\kappa+n}\Big)^{\alpha}\!\!\!\sum_{\underset{ \text{ s.t.}\  \pi_0\preceq\tilde{\pi}\preceq \pi}{\tilde{\pi}=(\tilde{B}_j)_{j\in J}}}\!\!\!(-1)^{|J|-k}\prod_{i=1}^{k}\Big(\card(\{j\in J,\ \tilde{B}_j\subset B_i\})-1\Big)!\prod_{j\in J}\Big(1-\frac{|\tilde{B}_j|}{n+\kappa}\Big)^{-\alpha}.
$$
Let us note that 
 $(1-\frac{j}{n+\kappa})^{-\alpha}$ is the $j$-th moment $m_{j}$ of the random variable 
$Y=\exp(\frac{Z}{n+\kappa})$ where $Z$ denotes a Gamma$(\alpha,1)$-distributed
random variable. Thus  $\displaystyle{\Pd_{\pi_{0}}(\mathcal{C}_{\alpha}\preceq\pi%
)=\frac{m_{n}}{\prod_{i=1}^{k} m_{\left|  B_{i}\right|  }}\un_{\{\pi_0 \preceq\pi\}}}.$  
Let $c_n$ denote the $n$-th cumulant of $Y$.  Formula \eqref{sgr_alpha} and the expression of cumulants  in terms of moments (see formula (1.30) in \cite{PitmanLN} for example) yield 
$\Pd_{\pi_{0}}(\mathcal{C}_{\alpha}=\{X\})=\frac{c_{n}}{m_{n}}$. 
\end{example}
\subsection{Loop clusters included in a subset of $X$}
For a subset  $D$ of $X$, let $\mathcal{DL}%
_{\alpha}^{(D)}$ denote the loops of $\mathcal{DL}_{\alpha}$ contained in $D$ and $\mathcal{C}_{\alpha}^{(D)}$  the partition of $D$ induced by $\mathcal{DL}%
_{\alpha}^{(D)}$. In general  $\mathcal{C}_{\alpha}^{(D)}$ is finer than the restriction of
$\mathcal{C}_{\alpha}$ to $D$ but coincides with it if $D$ is a union of
elements of $\mathcal{C}_{\alpha}$. \\
For every partition $\pi$ of $X$ with $k$ non-empty blocks $B_1,\ldots,B_k$, 
the loop sets $\La^{(B_1)}$, \ldots, $\La^{(B_k)}$, ${\La\setminus(\cup_{i=1}^{k}\La^{(B_i)})}$ are independent.  This yields the following equality:
\[
\Pd(\mathcal{C}_{\alpha}=\pi)=\Pd(\mathcal{C}_{\alpha
}\preceq\pi)\prod_{i=1}^{k}\Pd(\mathcal{C}_{\alpha}^{(B_{i})}=\{B_{i}\}).
\]
Let us note also that  if $U\subset D\subset X$ then
\[
\Pd(\mathcal{C}_{\alpha}\preceq\{U, U^c\})=\Pd(\mathcal{C}%
_{\alpha}^{(D)}\preceq\{U,D\setminus U\})\Pd(\nexists \ell\in\DL_{\alpha
}\;\text{visiting}\;U\;\text{and}\; D^c)
\]
(see \cite{Lawlerarx11} page 2 for a related result in Schramm-Loewner Evolution (SLE) context). 
The second term on the right-hand side equals 
\[e^{-\alpha\big(\mu(\DL(X))-\mu(\DL(U^c))-\mu(\DL(D))+\mu(\DL(D\setminus U))\big)};\]
If $X$ is finite, it  can be written  as
$$\Big(\frac{\det(G^{(U^c)})\det(G^{(D)}%
)}{\det(G)\det(G^{(D\setminus U)})}\Big)^{\alpha}\text{ or, with Jacobi's identity, as }\Big(\frac{\det(G_{|U\times U})\det
(G_{|D^c\times D^c})}{\det(G_{|(U\cup D^c)\times(U\cup D^c)})}\Big)^{-\alpha}.$$
\subsection{Computation of the semigroup using exit distributions}
The pro\-ba\-bility that $\Ca$ is finer than a partition $\pi$ can also be expressed using  exit distributions. The formula obtained is easier to use than \eqref{semigroupdetG} for graphs such as tori,  trees \ldots\ Let us first introduce some notations. 
For a subset $D$ of $X$, let $\partial D$ denote the inner boundary of $D$: \[\partial D=\{x\in D, C_{x,y}>0\;\text{for some}\;y\in  D^c\}\] and let $H^{(D)}$ denote the exit distribution (or Poisson kernel) from $D$: for $x\in X$ and $y\in D^c$, 
$H^{(D)}_{x,y}$ is the probability that a Markov chain with transition matrix $P$ starting from $x$ exits from $D$ at $y$. 
\begin{lem}
\label{semigroupexit}
Let $\pi=(B_i)_{i=1,\ldots,k}$ be a partition of $X$. Let $B=\cup_{j=1}^{k}\partial B_j$ denote the union of the boundary points of the blocks of $\pi$. Let  $\mathcal{H}^{(\pi)}$ denote the  matrix indexed by $ B$  defined by: 
\[\mathcal{H}^{(\pi)}_{x,y}=\left\{\begin{array}{ll}\un_{\{x=y\}} &\text{if}\; x,y\in \partial B_i\\
-H^{(B_i)}_{x,y}   &\text{if}\; x\in \partial B_i, y\in \partial B_j \;\text{and}\; i\neq j.\\  
 \end{array}\right.\]
If $B$ is finite then the probability for $\mathcal{C}_{\alpha}$ to be  finer than $\pi$ is 
 \[\Pd(\mathcal{C}_{\alpha}\preceq \pi)=\det(\mathcal{H}^{(\pi)})^{\alpha}.\] 
\end{lem}
\begin{proof}\
\begin{itemize}
 \item 
First, let us assume that $X$ is finite. Let $K$ denote the product of the block diagonal matrix $\diag(G^{(B_i)}, i\in\{1,\ldots,k\})$ by the matrix $G^{-1}$. 
We can rewrite the expression of $\Pd(\Ca\preceq\pi)$ given by Lemma \ref{semigrouppartition} as  $\Pd(\Ca\preceq\pi)=\det(K)^{\alpha}$. 
The restriction of $K$ to  $B_{i}\times B_{i}$ is the identity.  The  exit distribution from a subset $D$ verifies:  $H^{(D)}_{x,y}=\sum_{z\in D}G^{(D)}_{x,z}C_{z,y}$ for every $x\in D$ and $y\in D^c$. Therefore,  
\begin{equation}\label{detK}
 \Pd(\Ca\preceq\pi)=\det(K)^{\alpha}
\end{equation}
where 
\[
K_{x,y}=\left\{\begin{array}{ll}\un_{\{x=y\}} &\text{ if } x,y\in  B_i\\
-H^{(B_i)}_{x,y}   &\text{ if } x\in  B_i,\; y\in  B_j \text{ and } i\neq j.\\  
 \end{array}\right.
\]
Let $(\xi_n)_n$ denote a Markov chain with transition matrix $P$.  The trace  of $(\xi_n)_n$ on  $ B$ defines a Markov chain denoted by $(\tilde{\xi}_{n})$ on $ B$ and thus a Poisson point  process $\mathcal{D}\tilde{\mathcal{P}}$ on $\RR_{+}\times \mathcal{DL}( B)$ (see \cite{stfl} chap. 7): the discrete loops set at time $\alpha$ is $$\mathcal{D}\tilde{\mathcal{L}}_{\alpha}:= \{\ld_{|B},\; \ell \in \La \setminus \La^{(B^c)}\}. $$ 
Let $\tilde{\mathcal{C}}_{\alpha}$ be the partition of $ B$ induced by $\mathcal{D}\tilde{\mathcal{L}}_{\alpha}$. The non-empty subsets $\partial B_i$, ${i\in\{1,\ldots,k\}}$  define a partition of $B$ denoted by   $\partial \pi$. 
As $\{\Ca\preceq \pi\}$ is the event 
`\emph{$\forall i,j\in\{1,\ldots,k\}$ such that $i\neq j$, $\forall x\in \partial B_i$, $y\in \partial B_j$,\ $\{x,y\}$ is not crossed by a loop of $\DL_{\alpha}$ at time $\alpha$}', it only depends on the restriction of the loops on  $ B$, hence 
$\Pd(\Ca\preceq\pi)=\Pd(\tilde{\mathcal{C}}_\alpha\preceq\partial\pi)$. \\

For a subset $F$ of $B$, let $\tilde{H}^{(F)}$ denote the exit distribution from $F$ for $(\tilde{\xi}_{n})$. 
It follows from formula  \eqref{detK} applied to  $\tilde{\mathcal{C}}_{\alpha}$ that  
\[
\Pd(\Ca\preceq\pi)=\det(\tilde{K})^{\alpha}\] where \[\tilde{K}_{x,y}=\left\{\begin{array}{ll}\un_{\{x=y\}} &\text{if}\; x,y\in  \partial B_i\\
-\tilde{H}^{(\partial B_i)}_{x,y}   &\text{if}\; x\in  \partial B_i,\; y\in  \partial B_j \;\text{and}\; i\neq j.\\  
 \end{array}\right.
\]
To conclude, it remains to note that $\tilde{H}^{(\partial B_i)}_{x,y}=H^{(B_i)}_{x,y}$ for every $x\in \partial B_i$ and $y\in B\setminus \partial B_i$. 
\item 
Let us now assume that $X$ is a countable set. Let $(X_k)_k$ be an increasing sequence of finite subsets of $X$ 
such that $X=\cup_{k=1}^{+\infty}X_k$. 
As $B$ is assumed to be finite, there exists an integer $k_0$ such that $B$ is included in $X_{k_0}$. For $k\geq k_0$,  a loop in $\mathcal{DL}_{\alpha}^{(X_k)}$  that passes through two different blocks of $\pi$, passes through  two different blocks of $\pi_{|X_k}$ the restriction of $\pi$ to $X_k$. Therefore, the probability that $\mathcal{C}_{\alpha}$ is finer than $\pi$ is:
\[
1-\sup_{k\geq k_0}\Pd(\exists \ell\in \mathcal{DL}^{(X_k)}_{\alpha}\; \text{passing through two different blocks of}\;\pi)\\=\inf_{k\geq k_0}\Pd(\mathcal{C}_{\alpha}^{(X_k)}\preceq \pi_{|X_k}).
 \]
 By the first part of the proof, 
$\Pd(\mathcal{C}_{\alpha}^{(X_k)}\preceq\pi_{|X_k})=\det(\mathcal{H}^{(\pi_{|X_k})})^{\alpha}$. It remains to note that  the matrix $\mathcal{H}^{(\pi_{|X_k})}$ coincides with $\mathcal{H}^{(\pi)}$ for every $k\geq k_0$ to deduce that 
${\Pd(\mathcal{C}_{\alpha}\preceq \pi)=\det(\mathcal{H}^{(\pi)})^{\alpha}}$. 
\end{itemize}
\end{proof}
\subsection{Closed edges in a finite graph}
An edge is said to be \emph{closed} at time $\alpha$ if it is not crossed by  a loop of $\La$.  More generally, a family of edges $E$ is said to be closed at time $\alpha$ if every edge of $E$ is closed at time $\alpha$. \\
Let us assume that $\mathcal{G}$ is a finite graph. To compute the probability for a family of edges $E$ to be closed, we modify the conductances and the killing measure so that the conductance of every edge in $E$ is $0$ and the measure $\lambda$ is unchanged: $\tilde{C}_{x,y}=C_{x,y}\un_{\{\{x,y\}\not \in E\}}$ $\forall \{x,y\}\in \mathcal{E}$ and $\tilde{\kappa}_{x}=\kappa_x+\sum_{y,\ \{x,y\}\in E}C_{x,y}$ $\forall x\in X$.  
Let  $G_{E}$ denote  the Green function associated with $\{\tilde{C}_{e},e\in\mathcal{E},\ \tilde{\kappa}_x, x\in X\}$.
\begin{lem} \label{closedset}
Let us assume that  $\mathcal{G}=(X,\mathcal{E})$ is a finite graph. 
The probability for a family of edges $E$ to be closed at time $\alpha$ is
$\Big(\frac{\det(G_{E})}{\det(G)}\Big)^{\alpha}$. 
\end{lem}
\begin{proof}
For a loop $\ell$ and an edge $e$, let  $N_{e}(\ell)$ denote the number of jumps of $\ell$ across $e$. Set $N_{E}(\ell)=\sum_{e\in E}N_e(\ell)$. 
\[\Pd(E \;\text{is closed  at time}\; \alpha)=\exp\Big(-\alpha\mu(\DL(X))+
 \alpha\mu(\ell,  N_E(\ell)=0)\Big).\]
Recall that $\mu(\DL(X))=-\log(\det(I-P))$. 
For $s\in [0, 1]$, let $P_{E}(s)$ denote  a perturbation of $P$  defined by \[[P_E(s)]_{x,y}=\left\{\begin{array}{ll}sP_{x,y} & \text{if}\; \{x,y\}\in E\\
P_{x,y}  & \text{if}\; \{x,y\}\not\in E.                                      
\end{array}\right.
 \]
Similarly, $\mu(s^{N_E})=-\log(\det(I-P_E(s)))$. In particular, 
\[{\mu(\ell,~N_E(\ell)=0)=-\log(\det(I-P_E(0)))}.\] 
Lemma \ref{closedset} follows since $P_E(0)$ is the transition matrix associated with  ${\{\tilde{C}_{e},e\in\mathcal{E},\ \tilde{\kappa}_x, x\in X\}}$.
\end{proof}
We can deduce from Lemmas \ref{closedset} and \ref{semigrouppartition} that if $E$ is a set of edges of a finite graph with extremities in
different blocks of a partition $\pi=\{B_i,\; i\in I\}$ then 
\[
\Pd(\mathcal{C}_{\alpha}\preceq\pi\mid E\;\text{is closed at time}\;\alpha)=\Big(\frac{\prod_{i\in I}\det(G^{(B_{i})})}%
{\det(G_{E})}\Big)^{\alpha}.%
\]
\section{Renewal processes}
On the graph  $\ZZ$, the clusters $\mathcal{C}_{\alpha}$ are intervals between
closed edges at time $\alpha$ (namely edges which are not crossed by any loop of
$\mathcal{DL}_{\alpha}$). The Poisson loop sets  induced by a simple random walk killed at constant rate $\kappa$  have the following properties:
\begin{proposition}
Let us consider the graph $\ZZ$ endowed with unit conductances and a uniform killing measure with intensity $\kappa$. Set 
\[\rho^{(\kappa)}=\log\Big(1+\frac{\kappa}{2}+\sqrt{\kappa+\frac{\kappa^{2}}{4}}\Big).\] 
\label{proprenewal}
\begin{itemize}
\item[(i)] The midpoints of the closed edges at time $\alpha$ form a renewal
process. Moreover, 
\begin{itemize} 
\item  the probability that $\{n,n+1\}$ is closed at time $\alpha$ is equal to $(1-e^{-2\rho^{(\kappa)}})^{\alpha}$ for every $n\in \ZZ$;
\item given that  $\{0,1\}$ is closed at time $\alpha$,  the probability that $\{n,n+1\}$ is also closed  is equal to 
$\Big(\frac{1-e^{-2\rho^{(\kappa)}}}{1-e^{-2\rho^{(\kappa)}(n+1)}}\Big)^{\alpha}$ for every $n\in \NN$. 
\end{itemize} 
\item[(ii)] Assume that $\alpha\in]0, 1[$.  Let $\nu^{(\kappa)}$ denote the law of this renewal process i.e. the
law of the distance between two consecutive closed edges at time $\alpha$. For $\epsilon>0$ let $(Y_{\epsilon,i})_{i\in\NN^*}$ denote a sequence of independent random variables with distribution $\nu^{(\kappa\epsilon)}$. \\
For every $t>0$, as $\epsilon$ converges to $0$, ${\displaystyle \sqrt{\epsilon}\sum_{i=1}^{[\epsilon^{-\frac{1-\alpha}{2}}t]}Y_{\epsilon,i}}$ 
converges in law to the value at $t$ of a subordinator with potential density 
$\Big(\frac{2\sqrt{\kappa}}{1-e^{-2u\sqrt{\kappa}}}\Big)^{\alpha}$.
\end{itemize}
\end{proposition}
\begin{remark}
 This subordinator is associated with the Poisson covering (cf.~\cite{BertoinLN}, chapter~7) defined by the infinite measure on $\RR^{+}$ with density 
\[u\mapsto-\alpha\frac{d^2}{du^2}\log(1-e^{-2u\sqrt{\kappa}})=\frac{4\kappa\alpha e^{-2u\sqrt{\kappa}}}{(1-e^{-2u\sqrt{\kappa}})^2}.
\] These covering intervals can be viewed as images of the ``loop soup'' of intensity $\alpha$ associated with the Brownain motion killed at constant rate $\kappa$. 
\end{remark}
\begin{proof}[Proof of Proposition \ref{proprenewal}]\
\begin{itemize}
 \item[(i)] An edge $\{x-1,x\}$ is closed at time $\alpha$ if and only if   ${\mathcal{DL}_{\alpha}=\mathcal{DL}_{\alpha}^{(x+\NN)}\cup\mathcal{DL}_{\alpha}^{(x-1-\NN)}}$. The next closed edge is the first edge which is not crossed by any loop of
$\mathcal{DL}_{\alpha}^{(x+\NN)}$, and previous closed edges are defined
by $\mathcal{DL}_{\alpha}^{(x-1-\NN)}$ which is independent of
$\mathcal{DL}_{\alpha}^{(x+\NN)}$. Stationarity is obvious
as $\mathcal{DL}_{\alpha}^{(x+\NN)}-x$ is distributed like $\mathcal{DL}%
_{\alpha}^{(\NN)}$. 
For $n\in \NN$, let $r^{(\kappa)}_n$ denote the probability that $\{n,n+1\}$ is closed at time $\alpha$. By Lemma \ref{semigroupexit}, ${r^{(\kappa)}_n=(1-H^{(n+1+\NN)}_{n+1,n}H^{(n-\NN)}_{n,n+1})^{\alpha}}$ where  $H^{(D)}_{x,y}$ denote the probability that a simple random walk killed at constant rate $\kappa$ and starting at $x$ exits $D$ at point $y$. The functions $x\mapsto H^{(D)}_{x,y}$ can be computed from the solutions of  equation
\[(2+\kappa)u(x)-u(x+1)-u(x-1)=0,\]
 which  are linear combinations of
$\exp(\rho^{(\kappa)} x)$ and $\exp(-\rho^{(\kappa)} x)$. \\
We obtain $H^{(n+1+\NN)}_{n+k,n}=H^{(n-\NN)}_{n+1-k,n+1}=\exp(-\rho^{(\kappa)} k)$ for every $k\in \NN$. \\
For $n\in\NN^*$, let $q^{(\kappa)}(n)$ denote the probability that $\{n,n+1\}$ is closed at time $\alpha$ given that $\{0,1\}$ is closed at time $\alpha$. 
To compute $q^{(\kappa)}(n)$, we consider   a simple random walk  on $\NN$ killed at rate $\kappa$ and at point $0$; we denote it  $(\zeta_{k})_k$. As $(\La^{(\NN)})_{\alpha}$ is the Poisson loop sets associated with  $(\zeta_{k})_k$,  by Lemma \ref{semigroupexit}  we obtain: 
${q^{(\kappa)}(n)=(1-h_{n}^{(\enum{0}{n})}h_{n+1}^{(n+1+\NN)})^{\alpha}}$ where  $h^{(D)}$ is the exit distribution from a subset $D$ of $\NN$ for $(\zeta_{k})_k$:  \[h_{k}^{(\enum{0}{n})}=\frac{\sinh(\rho^{(\kappa)} k)}{\sinh(\rho^{(\kappa)} (n+1))}\quad\text{for}\quad 1\leq k\leq n+1\] and $h_{k}^{(n+1+\NN)}=e^{\rho^{(\kappa)}(n-k)}$ for  $k\geq n$.
Therefore 
 \[q^{(\kappa)}(n)=\Big(1-\frac{\sinh(\rho^{(\kappa)} n)}{\sinh(\rho^{(\kappa)}
(n+1))}e^{-\rho^{(\kappa)}}\Big)^{\alpha}=\Big(\frac{1-e^{-2\rho^{(\kappa)}}}{1-e^{-2\rho^{(\kappa)}(n+1)}}\Big)^{\alpha}.\] 
\item[(ii)] Let $I_{\kappa}$ denote the Laplace transform of the function $u\mapsto \Big(\frac{2\sqrt{\kappa}}{1-e^{-2\sqrt{\kappa}u}}\Big)^{\alpha}$:  
\[I_{\kappa}(s)= \int_{0}^{\infty}\Big(\frac{2\sqrt{\kappa}}{1-e^{-2\sqrt{\kappa
}u}}\Big)^{\alpha}e^{-su}du\quad  \text{for}\quad s>0.\] 
Let $\widehat{\nu^{(\kappa)}}$ denote the Laplace transform of $\nu^{(\kappa)}$. 
We shall prove that for every $s>0$,  $-\varepsilon^{\frac{\alpha-1}{2}}\log\big(\widehat{\nu^{(\kappa
\varepsilon)}}(s\sqrt{\varepsilon)}\big)$ converges
towards $1/I_{\kappa}(s)$  as $\varepsilon$ tends to $0$ which yields (ii).\\ 
As $q^{(\kappa)}(n)=\sum_{k=1}^{\infty}(\nu^{(\kappa)})^{\ast k}(n)$ for every $n\in\NN^*$, we have  the Laplace
transforms identity: $\widehat{\nu^{(\kappa)}}=\frac{\widehat{q^{(\kappa)}}}{1+\widehat{q^{(\kappa)}}}$. 
Therefore, it is sufficient to show that $\epsilon^{(1-\alpha)/2}\widehat{q^{(\kappa\epsilon)}}(s\sqrt{\epsilon})$ converges to $I_{\kappa}(s)$ as $\epsilon$ tends to $0$. 
Let $\kappa_1$ and $\kappa_2$ be two positive reals such that $0<\kappa_1<\kappa<\kappa_2$. For $\epsilon$ small enough, $\sqrt{\epsilon\kappa_1}\leq \rho^{(\kappa\epsilon)}\leq \sqrt{\epsilon\kappa_2}$ and $2\sqrt{\epsilon\kappa_1}\leq 1-e^{-2\rho^{(\kappa\epsilon)}}\leq 2\sqrt{\epsilon\kappa_2}$. 
For $a>0$ let  $f_{a}$  be the function defined by $f_a(x)=(1-e^{-2ax})^{-\alpha}$ for $x>0$. Therefore, for $\epsilon$ small enough, 
 \[(2\sqrt{\epsilon\kappa_1})^\alpha f_{\kappa_{2}}(\sqrt{\epsilon}(n+1))\leq q^{(\kappa\epsilon)}(n)\leq (2\sqrt{\epsilon\kappa_2})^\alpha f_{\kappa_{1}}(\sqrt{\epsilon}(n+1))\] for every $n\in \NN$. 
 By Lebesgue's dominated convergence theorem  applied to the function $g_{a,s,\epsilon}$
defined by \[g_{a,s,\epsilon}(x)=f_{a}(\sqrt{\epsilon}\lceil \frac{x}{\sqrt{\epsilon}}\rceil )e^{-s\sqrt{\epsilon}\lfloor \frac{x}{\sqrt{\epsilon}}\rfloor}\ \forall x>0,\]  we obtain:
\[\sqrt{\epsilon}\sum_{n=0}^{+\infty}f_{a}(\sqrt{\epsilon}(n+1))e^{-s\sqrt{\epsilon}n}\underset{\epsilon \rightarrow 0}{\rightarrow} \int_{0}^{+\infty}f_{a}(x)e^{-sx}dx\]
(for $\epsilon<\frac{1}{2a}$, $g_{a,s,\epsilon}$ is dominated by $x\mapsto(\frac{2}{ax})^{\alpha}\un_{]0, \frac{1}{2a}]}(x)+\frac{e^{-sx}}{(1-e^{-1/2})^{\alpha}}\un_{]\frac{1}{2a},+\infty[}(x)$). 
This implies that for every $s>0$
\[\underline{\lim}\; \epsilon^{(1-\alpha)/2}\widehat{q^{(\kappa\varepsilon)}}(s\sqrt{\varepsilon})\geq \int_{0}^{+\infty}\Big(\frac{2\sqrt{\kappa_1}}{1-e^{-2\sqrt{\kappa_2}u}}\Big)^{\alpha}e^{-su}du 
\]
and
\[\overline{\lim}\; \epsilon^{(1-\alpha)/2}\widehat{q^{(\kappa\varepsilon)}}(s\sqrt{\varepsilon})
\leq \int_{0}^{+\infty}\Big(\frac{2\sqrt{\kappa_2}}{1-e^{-2\sqrt{\kappa_1}u}}\Big)^{\alpha}e^{-su}du.\]
These inequalities hold for every $0<\kappa_1<\kappa<\kappa_2$. Therefore for every $s>0$, 
$\epsilon^{(1-\alpha)/2}\widehat{q^{(\kappa\epsilon)}}(s\sqrt{\epsilon})$ 
converges to $I_{\kappa}(s)$ which ends the proof of (ii). 
\end{itemize}
\end{proof}
\begin{remark} In the case of the simple random walk on $\NN$ killed at $0$, a
similar argument (detailed in \cite{LeJanCRAS12}) shows that:
\begin{itemize}
\item For $0<\alpha<1$, the midpoints of the closed edges at time $\alpha$ form a renewal process with holding times $(Y^{(\alpha)}_{n})_{n\geq 1}$; The generating function of  $Y^{(\alpha)}_{n}$  
 is $1-\frac{s}{\text{Li}_{\alpha
}(s)}$ where Li denotes the polylogarithm: $\forall\ |s|<1$, 
$\text{Li}_\alpha(s)=\sum_{k=1}^{+\infty}\frac{s^k}{k^{\alpha}}$. Set $S^{(\alpha)}_{n}=\sum_{i=1}^{n}Y^{(\alpha)}_{i}$ for $n\geq 1$. 
 As $\varepsilon$ tends to $0$,
${(\varepsilon S_{\lfloor\varepsilon^{\alpha-1}t\rfloor}^{(\alpha)},t\geq 0)}$ converges in
law towards a stable subordinator $(S_{t}^{(\alpha)},t\geq0)$ with index
$1-\alpha$. In the case of a finite interval $[0,  L]$, we obtain a renewal
process conditioned to jump at point $L$.
\item For $\alpha>1$, there are only a finite number of clusters. In
particular, $\Pd(S_{1}^{(\alpha)}=\infty)=\frac{1}{\zeta(\alpha)}$.
\end{itemize}
\end{remark}
\section{Bernoulli percolation and loop percolation}
\subsection{Bernoulli percolation}
Let $s=\{s_{e},\; e\in\mathcal{E}\}$ be a family of coefficients in $[0, 1]$. 
In the Bernoulli percolation model of parameter $s$ 
on the graph $\mathcal{G}=(X,\mathcal{E})$, every  edge $e$  is, independently of each
other, called `open' with probability $s_e$ and `closed' with probability $1-s_e$.  Vertices connected by open paths define a partition of $X$ denoted by $\mathcal{P}(s)$. 
We can compare $\mathcal{P}(s)$ to the partition induced by the set of primitive discrete loops of length $2$ for a finite graph $(X,\mathcal{E})$. 

\begin{lem}\label{bernpercolation}
Let us consider a finite graph $(X,\mathcal{E})$.
For $\alpha>0$ and $e=\{x,y\}\in \mathcal{E}$, set  $s_{\alpha,e}=1-(1-P_{x,y}P_{y,x})^{\alpha}$. \\ 
 The partition induced by the set of primitive discrete loops of length $2$ at time $\alpha$ has the same law as $\mathcal{P}(\{s_{\alpha,e},\; e\in \mathcal{E}\})$. 
\end{lem}
\begin{proof}
Let $\mathcal{PA}_{\alpha}$ denote the set of primitive discrete loops of length $2$ at time $\alpha$. The law  $\mathcal{PA}_{\alpha}$ for every $\alpha$ is a product measure $\nu$ on $\{0,1\}^{\mathcal{E}}$: for every edge $\{x,y\}\in\mathcal{E}$, let $\eta_{x,y}$ denote the class of the based loop $(x,y)$. The probability that a loop in $\mathcal{DL}_{\alpha}$ has $\{x,y\}$ as support  is \[\nu(\omega,\ \omega_{\{x,y\}}=1)=1-\exp\big(-\alpha\mu(\ell,\; \pi\ld=\eta_{x,y})\big) \]
and \[\mu(\ell,\; \pi\ld=\eta_{x,y})=\sum_{k=1}^{+\infty}\frac{1}{k}(P_{x,y}P_{y,x})^k = - \log(1-P_{x,y}P_{y,x}).\]
\end{proof}
 Bernoulli percolation clusters  on a graph appear to be a limiting case of
partitions defined by loop clusters in which only two points loops contribute
asymptotically. 

\begin{proposition}[\cite{LeJanCRAS12}]
Let us consider a finite graph $\mathcal{G}=(X,\mathcal{E})$ endowed with unit conductances and a uniform killing measure with intensity $\kappa>0$.
Let $\mathcal{C}^{(\kappa)}_{\alpha}$ be the partition induced by the Poisson loop set on $\mathcal{G}$ at time $\alpha$.   Fix $u>0$.\\
If  $\kappa$ and $\alpha$ tend to $+\infty$ such that  $\frac{\alpha}{\kappa^{2}}$ converges to   $u$, then  $\mathcal{C}^{(\kappa)}_{\alpha}$ converges in law towards the Bernoulli percolation of parameter $1-e^{-u}$. 
\end{proposition}
\begin{proof}
For a partition  $\pi=\{B_i,\; i\in I\}$ of $X$,  let $L(\pi)$ denote the set of edges of $\mathcal{E}$ linking different blocks 
of $\pi$. 
The law of the Bernoulli percolation of parameter $1-e^{-u}$ denoted by $\mathcal{P}(1-e^{-u})$  is
characterized by the identities:%
\[
\Pd(\mathcal{P}(1-e^{-u})\preceq\pi)=e^{-u\left|  L(\pi)\right|} \text{ for every  partition }\pi \text{ of }X.   %
\]
To prove the convergence of $\Pd(\Ca^{(\kappa)}\preceq\pi)$,  we apply Lemma \ref{semigroupexit}: $\Pd(\Ca^{(\kappa)}\preceq\pi)=\det(\mathcal{H}^{(\pi)})^{\alpha}$  where $\mathcal{H}^{(\pi)}$ is defined as in Lemma \ref{semigroupexit}. 
We note then that $\mathcal{H}^{(\pi)}_{x,y}$ is equivalent to $\kappa^{-1}$ if
$\{x,y\}$ belongs to $L(\pi)$ and  of order less or equal to
$\kappa^{-2}$\ otherwise. Indeed, if $x\in \partial B_{i}$ and $y\in \partial B_{j}$ for $i\neq j$,
$H^{(B_{i})}_{x,y}=\sum_{z\in B_{i}}P_{x,z}H^{(B_{i})}_{z,y}+P_{x,y}$
and $P_{i,j}\leq\frac{1}{\kappa}$ for all $(i,j)\in X\times X$.\\
A second-order Taylor expansion shows that $\log(\det(\mathcal{H}^{(\pi)}))=Tr(\log(\mathcal{H}^{(\pi)}))$ is equivalent to $-\frac{1}{2}Tr(Q^{2})$, with
$Q_{x,y}=\kappa^{-1}\un_{\{\{x,y\}\in L(\pi)\}}$.
\end{proof}
\subsection{Loop percolation on $\ZZ^{d}$ with $d\geq 2$}
Let us consider the Poisson loop process induced by the simple random walk on $\ZZ^{d},\;d\geq 2$,  killed at a constant rate
$\kappa>0$: $P_{x,x+u}=\frac{1}{2d+\kappa}$ for every $x\in\ZZ^{d}$ and $u\in\{\pm 1\}^d$. 
 Let $\theta(\alpha,\kappa)$ denote the probability of percolation at time $\alpha$ \textit{i.e.} the probability 
of any fixed point to be connected to infinity by an open path at time $\alpha$. 
The following Proposition presents some properties of the function $(\alpha,\kappa)\mapsto \theta(\alpha,\kappa)$: 

\begin{proposition}\label{probapercol}
Let $p_{c}$ denote the critical probability for  bond percolation on $\ZZ^d$ ($d\geq 2$). 
\begin{itemize}
\item[(i)] $\theta(\alpha,\kappa)$ is an increasing function of $\alpha$ and a decreasing function of $\kappa$.
\item[(ii)] $\theta(\alpha,\kappa)>0$ for every $\alpha>0$ and $\kappa>0$ such that $(1-\frac{1}{(2d+\kappa)^2})^{\alpha}<1-p_{c}$. 
\item[(iii)] For any $\alpha>0$, $\theta(\alpha,\kappa)$ vanishes for $\kappa$ large enough.
\end{itemize}
\end{proposition}
\begin{proof}\ 
\begin{itemize}
 \item[(i)] $\theta(\alpha,\kappa)$ is an increasing function of $\alpha$ since $\alpha \mapsto \mathcal{DL}^{(\kappa)}_{\alpha}$ is increasing. \\
To show that $\theta(\alpha,\kappa)$ is  a decreasing function of $\kappa$, we use an independent thinning procedure. 
Let $\kappa_{1}>\kappa_{2}>0$. The corresponding measures on based loops satisfy: 
\[\mup^{(\kappa_j)}( \ld=(x_1,\ldots,x_k))=\frac{1}{k}\Big(\frac{1}{2d+\kappa_j}\Big)^{k}\un_{\{x_{i+1}-x_{i}\in\{\pm 1\}^d\; \forall i\in\enu{k}\}}\] 
for $j\in\{1,2\}$,  $k\geq 2$ and $x_1,\ldots,x_k\in \ZZ^d$ with the convention ${x_{k+1}=x_{1}}$. By erasing independently each based loop $\ld\in\DLp_{\alpha}^{(\kappa_{2})}$ of length $k\geq 2$ with probability $1-(\frac{2d+\kappa
_{2}}{2d+\kappa_{1}})^{k}$, we obtain a discrete loop set having the same distribution  as $\mathcal{DL}^{(\kappa_1)}_{\alpha}$. 
\item[(ii)] By Lemma \ref{bernpercolation}, the partition induced by the set of primitive discrete loops of length 2 at time $\alpha$ has the same law as the Bernoulli percolation with parameter ${1-(1-\frac{1}{(2d+\kappa)^2})^{\alpha}}$. It follows from  Bernoulli percolation on $\ZZ^d$ that if  $1-(1-\frac{1}{(2d+\kappa)^2})^{\alpha}>p_{c}$ then ${\theta(\alpha,\kappa)>0}$.  
\item[(iii)]
To prove that $\theta(\alpha,\kappa)$ vanishes for $\kappa$ large enough, we show that there exists a finite real $C_{\alpha}>0$ such that  any self-avoiding path  $x=(x_{1}, x_{2}, \ldots, x_{L})$ of  length $L\in2\NN^*$ is open at time $\alpha$ with probability less than
$(\frac{C_{\alpha}}{\kappa})^{L}$. We can then conclude with the usual path-counting argument: for every $L\in 2\NN^*$,  $\theta(\alpha,\kappa)$ is bounded above by the probability that there exists an open  self-avoiding path of length $L$ at time $\alpha$ starting from the origin,  hence \[\theta(\alpha,\kappa)\leq \limsup_{L\rightarrow +\infty}(2d)^L(\frac{C_\alpha}{\kappa})^{L}=0\quad\text{for}\quad \kappa>2dC_\alpha.\]
Let us first introduce some notations.
\begin{itemize} 
\item Let $\mathfrak{P}(2,\enu{L})$ consist  of partitions of $\enu{L}$ in which all blocks have at least two elements (the number of  blocks of such a partition  $\pi$  is denoted by $|\pi|$ and  blocks are denoted by $\pi_1$, $\pi_2$, ...). 
\item For a vertex $v$ and a loop $\ell\in \DL(\ZZ^d)$, let $N_{v}(\ell)$ denote the number of times $\ell$ passes through $v$: $N_{v}(\ell)=\sum_{i=1}^{k}\un_{\{u_i=v\}}$ if $\ell$ is the class of the based loop $(u_1, \ldots,u_k)$. 
\end{itemize}
Let us consider a self-avoiding path of length $L\in2\NN^*$ denoted by ${x=(x_{1}, x_{2}, \ldots, x_{L})}$   and let $\mathcal{E}_x$ denote the set of edges $\{x_{2i-1},x_{2i}\},\ 1\leq i\leq L/2$.   \\
 To be open, the edges of  $x$ have to be covered by
the edges of $N\leq L$ loops of $\mathcal{DL}_{\alpha}$. Among these loops,  those that cover at least one edge $e\in\mathcal{E}_x$ can be used to define a partition $\pi\in\mathfrak{P}(2,\{1,\ldots,L\})$ as follows: let $\ell_1$ be a loop in $\DL_{\alpha}$ covering edge $\{x_1,x_2\}$. The first block $\pi_1$ of $\pi$ consists of the indices of the endpoints of $e\in\mathcal{E}_x$ covered by $\ell_1$. If $\pi_1\neq \{1,\ldots,L\}$, let $j$ be the smallest integer $i$ such that $x_i$ is not in $\pi_1$ and let $\ell_2$ be a loop  covering  $\{x_{j},x_{j+1}\}$. The second block $\pi_2$ of $\pi$ is defined  as the set of indices of endpoints of  $e\in\mathcal{E}_x$ covered by $\ell_2$ that are not in $\pi_1$, and so on until all elements of $\enu{L}$ belong to a block written down (an example is presented in Figure \ref{path}). Therefore, 
\begin{align*}
\un_{\{x\;\text{open at time}\;\alpha\}}  &  
  \leq\sum_{\pi\in \mathfrak{P}(2,\enu{L})}\Big(\sum_{\substack{\ell_{1}%
, \ldots,\ell_{\left|  \pi\right|  }\in\mathcal{DL}_{\alpha}\\ \text{pairwise distinct}}}%
\prod_{j=1}^{\left|  \pi\right|  }\prod_{i\in\pi_{j}}\un_{\{N_{x_{i}}(\ell_{j})>0\}}\Big)\\
&\leq\sum_{\pi\in \mathfrak{P}(2,\enu{L})}\Big(\sum_{\substack{\ell_{1}%
, \ldots,\ell_{\left|  \pi\right|  }\in\mathcal{DL}_{\alpha}\\ \text{pairwise distinct}}}%
\prod_{j=1}^{\left|  \pi\right|  }\prod_{i\in\pi_{j}}N_{x_{i}}(\ell_{j})\Big).
\end{align*}
\begin{figure}[htb]
\fbox{%
\begin{minipage}{0.4\textwidth}
 \includegraphics[scale=0.5]{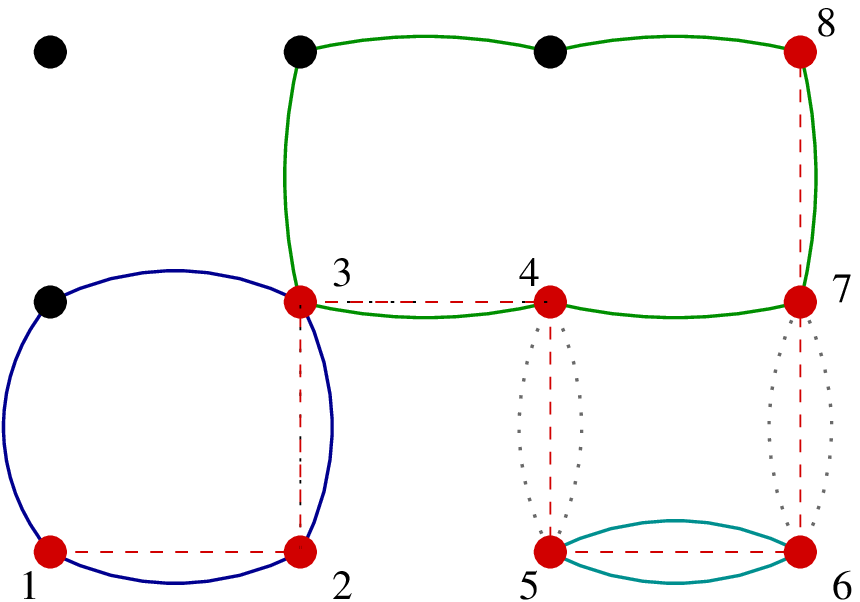}
\end{minipage}
\begin{minipage}{0.55\textwidth}
 Example of an open path $x=(1,2,3,4,5,6,7,8)$ with five loops covering its edges.  Three of them (drawn in solid lines) cover the set of edges $\mathcal{E}_x$; we associate to these three loops, the partition $\pi\in\mathfrak{P}(2,\{1,\ldots,8\})$ with three blocks $\pi_1=\{1,2\}$, $\pi_{2}=\{3,4,7,8\}$ and $\pi_{3}=\{5,6\}$.  
\end{minipage}}
\caption{\label{path}}
\end{figure}

By Campbell's
formula (equation (4.2) page 36 in \cite{stfl}), for every $k\in\NN^*$ and for every positive functions $F_1,
\ldots, F_k$
\[\Ed\Big(\sum_{\substack{\ell_1,\ldots,\ell_k\in\La,\\ \text{pairwise  distinct}}}\prod_{i=1}^{k}F_i(\ell_i)\Big)=\prod_{i=1}^{k}\Big(\int F_i(\ell)\alpha d\mu(\ell)\Big).\] Therefore,  %
\[\Pd(x\;\text{is open at time}\;\alpha)   \leq\sum_{\pi\in \mathfrak{P}(2,\enu{L})}\prod_{j=1}^{\left|  \pi\right|  }\int\Big(\prod_{i\in\pi_{j}}%
N_{x_{i}}(\ell)\Big)\alpha\mu(d\ell).\]
For a finite set $A$, let $\mathfrak{S}_A$ denote the set of permutations of elements of $A$. 
By Proposition 5 page 20 in \cite{stfl}, for every $k\geq 2$ and for every vertices $y_1,\ldots,y_k$,  \[\int \prod_{i=1}^{k}N_{y_i}(\ell)d\mu(\ell)=\frac{1}{k}\Big(\prod_{i=1}^{k}\lambda_{y_i}\Big)\!\!\sum_{\sigma\in \mathfrak{S}_{\enu{k}}}\!\!\!G_{y_{\sigma(1)},y_{\sigma(2)}}G_{y_{\sigma(2)},y_{\sigma(3)}}\ldots G_{y_{\sigma(k)},y_{\sigma(1)}}\]
Therefore,
$\displaystyle{\Pd(x\;\text{is open at time}\;\alpha)   \leq\Big(\prod_{i=1}^{L}\lambda_{x_i}\Big)S_{\alpha}(x)}$ where \begin{equation}
\label{eqSalpha}
S_{\alpha}(x)=\sum_{\pi\in \mathfrak{P}(2,\enu{L})}\alpha^{|\pi|}\prod_{j=1}^{\left|\pi\right|}\Big(\frac{1}{|\pi_j|}\sum_{\sigma\in \mathfrak{S}_{\pi_j}}G_{x_{\sigma(1)},x_{\sigma(2)}}G_{x_{\sigma(2)},x_{\sigma(3)}}\ldots G_{x_{\sigma(|\pi_j|)},x_{\sigma(1)}}\Big).
 \end{equation}
The blocks of a partition $\pi\in\mathfrak{P}(2,\enu{L})$ in \eqref{eqSalpha} can be seen as the orbits of a permutation without fixed point. Since a circular order on $k$ integers corresponds to $k$ different permutations of these integers, $S_{\alpha}(x)$ can be rewritten as follows: 
\begin{equation}
\label{per0}
S_{\alpha}(x)=\sum_{\sigma
\in\mathfrak{S}_{\enu{L}}^{0}}\alpha^{m(\sigma)}G_{x_{1},x_{\sigma(1)}}\ldots G_{x_{L},x_{\sigma(L)}}.
\end{equation}
where
\begin{itemize}
\item $m(\sigma)$ denotes the number of cycles in a permutation $\sigma$,
\item  $\mathfrak{S}_{\enu{L}}^{0}$ denotes the set of permutations of $\enu{L}$ without
fixed point (such a permutation corresponds to a configuration without isolated points).
\end{itemize}(See Lemma \ref{permanent} in the appendix for  details.)
The right-hand side of equality \eqref{per0}  is nothing but  $\Per_{\alpha}^{0}(G_{x_i,x_j},1\leq i,j\leq L)$ defined in \cite{stfl} page 41. 
To conclude we use that  \[\displaystyle{\Per^{0}_{1}(G_{x_i,x_j}, 1\leq i,j\leq L)\leq \prod_{i=1}^{L}\Big(\sum_{1\leq j\leq L,\ j\neq i}G_{x_i,x_j}\Big)\leq \Big(\sum_{y\in \ZZ^d\setminus\{0\}}G_{0,y}\Big)^L}\]  since the vertices $x_i$ are pairwise distinct. 
As $(P1)_u=\frac{2d}{2d+\kappa}$ for every $u\in\ZZ^d$,  \[\sum_{y\in \ZZ^d\setminus\{0\}}\lambda_yG_{0,y}=\sum_{k=1}^{+\infty}(P^k1)_0=\frac{2d}{\kappa}.\] Thus $\Pd(x\;\text{is open at time}\;\alpha)\leq (\frac{2d}{\kappa}\max(\alpha,1))^L$ which ends the proof. 
\end{itemize}
\end{proof}
\begin{remark}\
\begin{itemize}
 \item[(i)]
It follows from Proposition \ref{probapercol} that for every $\alpha>0$, there is a  finite value of $\kappa$, which we denote by  $\kappa^{c}(\alpha)$, above which $\theta(\alpha,\kappa)$
vanishes; Morever $\kappa^{c}$ is an increasing function that converges to $+\infty$  as $\alpha\rightarrow\infty$.
\item[(ii)] As the simple random walk on $\ZZ^2$ is recurrent, for $d=2$ the probability 
that a fixed edge is open at time $\alpha>0$ converges to $1$ as $\kappa$ tends to $0$. 
Indeed,  let us first note that  $\Pd(N_{x}^{(\alpha)}=0)=\Big(\frac{1}{\lambda_xG_{x,x}}\Big)^{\alpha}$ for every $x\in \ZZ^2$
(this equality has been proven for a finite graph\footnote{For a finite graph  $\mathcal{G}=(X,\mathcal{E})$, $\Pd(N_{x}^{(\alpha)}=0)$ is equal to  \[\exp\Big(-\alpha\Big(\mu(\DL(X))-\mu(\DL(X\setminus\{x\}))\Big)\Big)=\Big(\frac{\det(G^{(X\setminus\{x\})})}{\det(G)\lambda_x}\Big)^{\alpha}=(\lambda_xG_{x,x})^{-\alpha}.\]}, so we apply it to the restriction of the random walk to $\ZZ^2\cap[-M, M]^2$ and let $M$ tend to $+\infty$). Therefore $\Pd(N_{x}^{(\alpha)}>0)$  converges to $1$ as $\kappa$ tends to~$0$.\\
 Fix $x\in \ZZ^2$ and $u\in\{\pm 1\}^2$.  By symmetry, \[\Pd(N_{x}^{(\alpha)}>0)=\sum_{v\in\{\pm 1\}^{2}}\Pd(N_{x,x+v}^{(\alpha)}>0)=4\Pd(N^{(\alpha)}_{x,x+u}>0).\] 
 Since $\alpha\mapsto \mathcal{DL}_{\alpha}$ has independent stationary increments,  \[\Pd(N_{x,x+u}^{(\alpha)}=0)=\mathbb{P}(N_{x,x+u}^{(\alpha/n)}=0)^{n}\quad\text{for every}\quad n>0.\]
Therefore, 
\begin{equation}
\label{probaedgezd}
\Pd(N^{(\alpha)}_{x,x+u}>0)=1-\Pd(N^{(\alpha/n)}_{x,x+u}=0)^{n}\geq 1-(1-\frac{1}{4}\Pd(N_{x}^{(\alpha/n)}>0))^{n}.
\end{equation}
For a fixed $\varepsilon>0$, let us choose $n$ such that $(1-\frac{1}{8})^n\leq \varepsilon$. 
There exists $\kappa_{\varepsilon}$ such that for $\kappa\leq \kappa_{\epsilon}$,   $\mathbb{P}(N_{x}^{(\alpha/n)}=0)\leq \frac{1}{2}$. It follows from \eqref{probaedgezd} that 
$\mathbb{P}(N_{x,x+u}^{(\alpha)}>0)\geq 1-\varepsilon$.
\end{itemize}
\end{remark}

\section{Complete graph}
This section is devoted to the study of Poisson loop sets on the complete graph $K_n$ 
endowed with unit conductances and a 
uniform killing measure. The set of vertices is identified with $\enu{n}$ and the set of partitions of $\enu{n}$ is denoted by $\mathfrak{P}(\enu{n})$.   
The intensity of the killing measure is denoted by $\kappa_n=n\epsilon_n$ with $\epsilon_n>0$, hence  
the coefficients of the transition matrix $P$ are: $P_{x,y}=\frac{\un_{\{x\neq y\}}}{\lambda^{(n)}}$ 
  with $\lambda^{(n)}=n-1+n\epsilon_n$ for every $x,y\in\enu{n}$. As $n$ will vary, the loop set and the partition defined 
by the loop clusters at time $\alpha$ will be denoted by $\mathcal{DL}_{\alpha}^{(n)}$ and $\mathcal{C}_{\alpha}^{(n)}$ respectively.    
In the first part, $n$ is fixed; We present another construction of the coalesecent process  
$(\mathcal{C}^{(n)}_{\alpha})_{\alpha\geq 0}$ and use this construction to define a similar coalescent process on the interval $[0, 1]$. 
In the second part, we let $n$ tend to $+\infty$ and  
describe the distribution of the first time when $\mathcal{C}^{(n)}_{\alpha}$ has no block 
of size one (cover time) and the first time when $(\mathcal{C}^{(n)}_{\alpha})_{\alpha}$ has only one block (coalescent time). 
\subsection{Another construction of the coalescent process}
\begin{proposition}\label{transitionKn}
From state $\pi=\{B_i,\; i\in I\}\in \mathfrak{P}(\enu{n})$, the only possible transitions of $(\mathcal{C}^{(n)}_{\alpha})_ {\alpha\geq 0}$ are to partitions $\pi^{\oplus J}$ where $J=\{j_1,\ldots,j_{L}\}$ is a subset of  $I$ with $L\geq 2$ elements. Its transition rate from $\pi$ to  $\pi^{\oplus J}$ is equal to: 
\begin{align}
\tau^{(n)}_{\pi,\pi^{\oplus J}}&=\sum_{k\geq L}\frac{1}{kn^k(\epsilon_n+1)^k}\sum_{(i_1,\ldots,i_k)\in W_k(J)} \prod_{u=1}^{k}|B_{i_u}|\label{eqtransitionKn1}\\
&=\sum_{k\geq L}\frac{1}{kn^k(\epsilon_n+1)^k}\sum_{\substack{(k_1,\ldots,k_L)\in(\NN^*)^{L},\\k_1+\ldots +k_{L}=k}}\binom{k}{k_1,\ldots,k_{L}}\prod_{u=1}^{L}|B_{j_u}|^{k_{u}}\label{eqtransitionKn2}
\end{align}
where $W_k(J)$ is the set of $k$-tuples of  $J$ in which each element of $J$ appears. 
\end{proposition}
\begin{proof}[Proof of Proposition \ref{transitionKn}]
Let us recall the expression of $\tau^{(n)}_{\pi,\pi^{\oplus J}}$ obtained in Example \ref{extransKn}: \[\tau^{(n)}_{\pi,\pi^{\oplus J}}=\sum_{I\subsetneqq J}(-1)^{|I|+1}\log\Big(1-\frac{1}{\lambda^{(n)}+1}\sum_{u\in J\setminus I}|B_u|\Big).\]
By expanding the logarithm,  we obtain:  
 \[\tau^{(n)}_{\pi,\pi^{\oplus J}}
=\sum_{k\geq 1}\frac{1}{k(\lambda^{(n)}+1)^k}\sum_{I\subsetneqq J}(-1)^{|I|}\Big(\sum_{u\in J\setminus I}|B_u|\Big)^k.
\]
To establish   
\begin{equation}
\label{cribleaux}
\sum_{I, I\subsetneqq J}(-1)^{|I|}\Big(\sum_{u\in J\setminus I}|B_u|\Big)^k=\sum_{(i_1,\ldots,i_k)\in W_k(J)} \prod_{u=1}^{k}|B_{i_u}|
\end{equation} which ends the proof of \eqref{eqtransitionKn1}, 
we make use of  the inclusion-exclusion principle. 
Let us consider the following random experiment: \emph{`we choose uniformly at random $k$ points in 
the set ${B:=\cup_{j\in J}B_j}$'} and for $j\in J$, let $E_j$ denote the event \emph{`at least one of the $k$ points falls into $B_j$'}. 
By the inclusion-exclusion principle,  the coefficient $\sum_{I\subsetneqq J}(-1)^{|I|}\Big(\sum_{u\in J\setminus I}|B_u|\Big)^k$ divided by $|B|^k$ is equal to  $\Pd(\cap_{j\in J} E_j)$. 
The probability $\Pd(\cap_{j\in J} E_j)$ can also be decomposed by introducing  
the events $A_{i,j}$   \emph{`the $i$-th point falls into  $B_j$'} for every $i\in\enu{k}$ and $j\in J$: $\Pd(\cap_{j\in J} E_{j}) =\sum_{(i_1,\ldots,i_k)\in W_k(J)}\Pd(\cap_{u=1}^{k}A_{u,i_u})$. This corresponds to the right-hand side of \eqref{cribleaux} divided by $|B|^k$. Equation \eqref{eqtransitionKn2} is obtained by rearranging the terms of the product
$\prod_{u=1}^{k}|B_{i_u}|$.  
\end{proof}

Formula \eqref{eqtransitionKn1} of the transition rate has a simple interpretation:  
if we choose $R$ points uniformly at random in $\enu{n}$ where $R$ has the logarithmic series 
distribution\footnote{The logarithmic series distribution with parameter $0<a<1$ is defined as  the measure \\
 ${\nu_n=\frac{1}{-\log(1-a)}\sum_{k=1}^{+\infty}\frac{a^k}{k}\delta_{k}}$.} with parameter $\frac{1}{\epsilon_n+1}$, the probability that at least one point falls into each block $\{B_i,\, i\in J\}$ and none outside of $\cup_{j\in J}B_j$ is equal to $(-\log(1-\frac{1}{\epsilon_n+1}))^{-1}\tau^{(n)}_{\pi,\pi^{\oplus J}}$. 
From this remark, we derive a simpler construction of $(\mathcal{C}^{(n)}_{\alpha})_{\alpha\geq 0}$:
\begin{proposition}
\label{construction}
 Let us define a $\mathfrak{P}(\enu{n})$-valued sequence $(Y_k)_k$ iteratively as follows:
\begin{itemize}
 \item $Y_0$ is a random variable with values in  $\mathfrak{P}(\enu{n})$;
\item Let $k\in \NN$. Given $Y_k$,  
\begin{itemize}
\item  we choose an integer $R\geq 2$ independent of $Y_0,\ldots,Y_n$, following the probability distribution  \[\nu=\frac{1}{\beta_{\epsilon_n}}\sum_{k\geq 2}\frac{1}{k(\epsilon_n+1)^k}\delta_{k}\quad \text{where}\quad \beta_{\epsilon_n}=-\log(1-\frac{1}{\epsilon_n+1})-\frac{1}{\epsilon_n+1};\]   
\item we choose $R$ points $U_1,\ldots,U_R$ uniformly at random on $\enu{n}$ and independently of $Y_0,\ldots,Y_k, R$; 
\item  $Y_{k+1}$ is obtained from $Y_{k}$ by merging  blocks of $Y_k$ that contain at least one of the $R$ points $U_1,\ldots, U_R$. 
\end{itemize}
\end{itemize}
Let $(Z_t)_{t\geq 0}$  be a Poisson process with intensity $\beta_{\epsilon_n}$ and independent of $(Y_k)_k$.\\
The process $(\Pi^{(n)}_t)_{t\geq 0}$ defined by $\Pi^{(n)}_t=Y_{Z_t}$ for every $t\geq 0$ 
is a continuous Markov chain:
\begin{itemize}
\item  The only possible transitions from state $\pi=(B_i)_{i\in I}$ are to partitions $\pi^{\oplus J}$ 
where $J$ is a subset of $I$ with at least two elements and its transition rate from 
$\pi$ to $\pi^{\oplus J}$  is  $\tau^{(n)}_{\pi,\pi^{\oplus J}}$ defined in Proposition \ref{transitionKn}. 
\item If $\pi$ is a partition of $\enu{n}$ with $k$ non-empty blocks $(B_1,\ldots,B_k)$ then \[\Pd_{\pi_{0}}(\Pi^{(n)}_{t}\preceq \pi)=(\frac{\epsilon_n}{\epsilon_n+1})^{t}\prod_{i=1}^{k}(1-\frac{|B_i|}{n(1+\epsilon_n)})^{-t}\un_{\{\pi_0 \preceq\pi\}}. \]
\end{itemize}
\end{proposition}
\subsection{A similar coalescent process on the interval $[0, 1]$}
The algorithm described in Proposition \ref{construction} can be adapted to define a  coalescent process $(\Pi_{t})_{t\geq 0}$ on the interval $[0, 1]$ such that for every partition $\pi$ of $[0,1]$ and every $t\geq 0$, $\Pd(\Pi^{(n)}_{t}\preceq \pi^{(n)})$ converges to $\Pd(\Pi_t\preceq \pi)$  if $\epsilon_n$ converges to $\epsilon>0$ and the partition $\pi^{(n)}$ of $\enu{n}$  converges to $\pi$ as $n$ tends to $+\infty$. 
\begin{proposition} 
For $k\in \NN^*$, let $\mathfrak{P}_k([0, 1])$ denote the set of partitions of $[0, 1]$ with $k$  blocks having a positive Lebesgue measure and let $\mathfrak{P}([0, 1])$ denote $\underset{k\in \NN^*}{\bigcup}\mathfrak{P}_k([0, 1])$. For a partition  $\pi=(b_i)_{i\in I}\in \mathfrak{P}([0, 1])$ and a subset  $J$ of $I$, let  $\pi^{\oplus J}$ denote the partition  obtained from $\pi$ by merging  blocks $b_i$ for $i\in J$.  Let $\epsilon$ be a positive real. 
 Let us define a $\mathfrak{P}([0, 1])$-valued sequence $(Y_n)_n$ iteratively as follows:
\begin{itemize}
 \item $Y_0$ is a random variable with values in  $\mathfrak{P}([0, 1])$;
\item Let $n\in\NN$. Given $Y_n$,  
\begin{itemize}
\item  we choose an integer $R\geq 2$ independent of $Y_0,\ldots,Y_n$ following the probability distribution \[\nu=\frac{1}{\beta_{\epsilon}}\sum_{k\geq 2}\frac{1}{k(\epsilon+1)^k}\delta_k\quad\text{where}\quad \beta_{\epsilon}=-\log(1-\frac{1}{\epsilon+1})-\frac{1}{\epsilon+1};\]   
\item we choose $R$ points $U_1,\ldots,U_R$ uniformly at random in the interval $[0, 1]$ and independently of $Y_0, \ldots, Y_n, R$;
\item $Y_{n+1}$ is obtained from $Y_{n}$ by merging  blocks of $Y_n$ that contain at least one of the $R$ points $U_1,\ldots, U_R$. 
\end{itemize}
\end{itemize}
Let $(Z_t)_{t\geq 0}$  be a Poisson process with intensity $\beta_{\epsilon}$ and independent of $(Y_n)_n$.\\
The process $(\Pi_t)_{t\geq 0}$ defined by $\Pi_t=Y_{Z_t}$ for every $t\geq 0$ is a continuous Markov chain:
\begin{itemize}
\item  its positive transition rates are from  a partition $\pi=(b_i)_{i\in I}$ to a partition $\pi^{\oplus J}$ with $J\subset I$ having at least two elements;  
the value of  such a transition rate  is   \[\tau_{\pi,\pi^{\oplus J}}:=\sum_{k\geq |J|}\frac{1}{k(\epsilon+1)^k}\sum_{\substack{(k_1,\ldots,k_{|J|})\in (\NN*)^{|J|},\\ k_1+\ldots+k_{|J|}=k}}\binom{k}{k_1,\ldots,k_{|J|}} \prod_{u=1}^{|J|}\leb(b_{u})^{k_u};
\]
\item if $\pi$ is a partition of $[0, 1]$ with $k$ non-empty blocks $(b_1,\ldots,b_k)$ then 
\[\Pd_{\pi_{0}}(\Pi_{t}\preceq \pi)=(\frac{\epsilon}{\epsilon+1})^{t}\prod_{i=1}^{k}(1-\frac{\leb(b_i)}{1+\epsilon})^{-t}\un_{\{\pi_0 \preceq\pi\}}. \]
\end{itemize}
\end{proposition}
\begin{proof}
By construction $(Y_n)_n$ is a $\mathfrak{P}([0, 1])$-valued Markov chain such that for every $n\in \NN$, $Y_{n+1}=Y_n$ or $Y_{n+1}$ is a coarser partition of $[0, 1]$ than $Y_n$ obtained by merging several blocks of $Y_n$ in one block. \\ 
 Let $\pi=(b_i)_{i\in I}$ be a partition of $[0, 1]$ and let $J$ be a subset of $I$ with at least two elements. 
Set $b_{J}=:\underset{i\in J}{\bigcup}b_i$. Given the event  $\{Y_{n}=\pi\}$, $Y_{n+1}$ is equal to $\pi^{\oplus J}$ 
if among the $R$ points $U_1,\ldots, U_R$ uniformly distributed in the interval $[0, 1]$ at least one point falls in every block $b_i$ for $i\in J$ and none falls in $[0, 1]\setminus b_{J}$. Therefore, $P(Y_{n+1}=\pi^{\oplus J}|Y_{n}=\pi)$ is equal to 
\begin{multline*}\sum_{k\geq |J|}^{+\infty}\nu(\{k\})\sum_{(i_1,\ldots,i_k)\in W_k(J)}\prod_{u=1}^{k}\leb(b_{i_u})\\
=\frac{1}{\beta_{\epsilon}}\sum_{k\geq |J|}^{+\infty}\frac{1}{k(\epsilon+1)^k}\sum_{\underset{\sum_{i}k_i=k}          {(k_1,\ldots,k_{|J|})\in (\NN^*)^{|J|},}}\binom{k}{k_1,\ldots,k_{|J|}}\prod_{i\in J}\leb(b_{i})^{k_i}. 
\end{multline*}
Given the event $\{Y_{n}=\pi\}$, $Y_{n+1}$ is equal to $\pi$ if  $U_1,\ldots, U_R$  fall in one block of the partition $\pi$.
Therefore, \begin{align*}P(Y_{n+1}=\pi |Y_{n}=\pi)=&\sum_{k= 2}^{+\infty}\nu(\{k\})\Big(\sum_{i \in I}\leb(b_{i})^k\Big)\\=&-\frac{1}{\beta_{\epsilon}}\Big(\frac{1}{\epsilon+1}+\sum_{i \in I}\log(1-\frac{\leb(b_{i})}{\epsilon+1})\Big). \end{align*}
The generator of $(\Pi_t)_{t\geq 0}$ is then $A=\beta_{\epsilon}(Q-I)$ where $Q$ is the transition matrix of $(Y_n)_n$. \\
Let us now  compute $\Pd_{\pi_0}(\Pi_t\preceq \pi)$ where $\pi$ is  a partition of $[0, 1]$ with $k$ non-empty blocks $(b_1,\ldots,b_k)$ coarser than or equal to $\pi_0$.  $\{\Pi_t\preceq \pi\}$ means that all points that are chosen simultaneously in the interval $[0, 1]$   before time $t$ belong to the same block of $\pi$. Therefore, by decomposing   $\{\Pi_t\preceq \pi\}$ according to the value of $Z_t$ and the number of points falling at the same time, we obtain that $\Pd_{\pi_0}(\Pi_t\preceq \pi)$ is equal to
\[e^{-\beta_{\epsilon}t}\Big\{1+\sum_{n=1}^{+\infty}\frac{(\beta_{\epsilon}t)^n}{n!}\!\!\!\!\!\sum_{r_1\geq 2,\ldots,r_n\geq 2} \prod_{i=1}^{n}\Big(\nu(\{r_i\})(\leb(b_{1})^{r_i}+\ldots + \leb(b_{k})^{r_i})\Big)\Big\}.\]
We expand the product \[\prod_{i=1}^{n}(\leb(b_{1})^{r_i}+\ldots + \leb(b_{k})^{r_i})=\sum_{(u_1,\ldots,u_n)\in\enu{k}^n}\leb(b_{u_1})^{r_1}\ldots \leb(b_{u_n})^{r_n}\] and use that 
\[\sum_{r\geq 2}\nu(\{r\})\leb(b_{u})^{r}=\frac{1}{\beta_{\epsilon}}\Big(-\log(1-\frac{\leb(b_u)}{\epsilon+1})-\frac{\leb(b_u)}{\epsilon+1}\Big)\] 
to obtain that $\Pd_{\pi_0}(\Pi_t\preceq \pi)$ is equal to
\[e^{-\beta_{\epsilon}t} \Big\{1+\sum_{n=1}^{+\infty}\frac{t^n}{n!}\sum_{(u_1,\ldots,u_n)\in \enu{k}^n}\prod_{i=1}^{n}\Big(-\log(1-\frac{\leb(b_{u_i})}{\epsilon+1})-\frac{\leb(b_{u_i})}{\epsilon+1}\Big)\Big\}.\]
The second sum is also equal to: \[\sum_{\underset{n_1+\ldots+ n_k=n}{(n_1,\ldots,n_k)\in \NN^k}}\binom{n}{n_1,\ldots,n_k}\prod_{i=1}^{k}\Big(-\log(1-\frac{\leb(b_{i})}{\epsilon+1})-\frac{\leb(b_{i})}{\epsilon+1}\Big)^{n_i}.\]
Therefore, $\Pd_{\pi_0}(\Pi_t\preceq \pi)$ is equal to 
\[e^{-\beta_{\epsilon}t} \Big\{1+\sum_{n=1}^{+\infty}\sum_{\underset{n_1+ \ldots + n_k=n}{(n_1,\ldots,n_k)\in \NN^k}}\prod_{i=1}^{k}\frac{t^{n_i}}{n_i!}\Big(-\log(1-\frac{\leb(b_{i})}{\epsilon+1})-\frac{\leb(b_{i})}{\epsilon+1}\Big)^{n_i}\Big\}\]
The expression inside the braces is equal to 
\[\exp\Big(t\sum_{i=1}^{k}\big(-\log(1-\frac{\leb(b_{i})}{\epsilon+1})-\frac{\leb(b_{i})}{\epsilon+1}\big)\Big).\]
This yields 
\[\Pd_{\pi_0}(\Pi_t\preceq \pi)=\exp\Big(t\log(1-\frac{1}{\epsilon+1})-t\sum_{i=1}^{k}\log(1-\frac{\leb(b_{i})}{\epsilon+1})\Big).\]
\end{proof} 
\subsection{Asymptotics for cover time}
A vertex $x$ is said to be \emph{isolated at time $\alpha$} if no loop $\ell\in \La^{(n)}$ passes through $x$. 
We call  \textit{`cover time'} the smallest $\alpha$ such that $K_n$ has no isolated vertex at time $\alpha$ 
 and we denote it $T_n$.\\
Let us assume that the intensity of the killing measure is proportional to the size of the graph: $\kappa_n=n\epsilon$ with $\epsilon>0$.   
If we use the algorithm  described in Proposition \ref{construction} to define $\Ca^{(n)}$, 
then we need to choose uniformly at random 
an average of $\alpha\beta_{\epsilon}\Ed(R)=\frac{\alpha}{\epsilon(\epsilon+1)}$ points in $\enu{n}$. 
If the points are drawn one by one  and not  by packs of random sizes, then 
the solution of the classical coupon 
collector's problem  provides that the values of $T_n$ is around ${\epsilon(1+\epsilon)n\log(n)}$ for large $n$.  This analogy holds: 
\begin{proposition}
\label{nbisolated}
Let us assume that the intensity of the killing measure on $K_n$ is $\kappa_n=n\epsilon$ with $\epsilon>0$. 
 For every $a\in\RR$, the number of isolated vertices at time $\alpha_n=\epsilon(1+\epsilon)n(\log(n)+a+o(1))$ converges in law to the Poisson distribution with parameter $e^{-a}$. 
In particular,  $\frac{T_n}{n\epsilon(1+\epsilon)}-\log(n)$ converges in law to the Gumbel distribution\footnote{The cumulative distribution function of the Gumbel distribution is $x\mapsto e^{-e^{-x}}$.}. 
\end{proposition}
\begin{proof}
 For each vertex $x$, let $I_{n,x}$ denote the indicator of the event `\emph{$x$ is isolated at time $\alpha_n$}'. The number of isolated vertices at time $\alpha_n$ is $S_n=\sum_{x=1}^{n}I_{n,x}$. For every $k\in \NN^*$, let $\Ed(S_n)_k:=\Ed(S_n(S_n-1)\ldots (S_n-k+1))$ denote its $k$-th factorial moment; $\Ed(S_n)_k$ is the sum of  $\Pd(I_{n,x_1}=\ldots=I_{n,x_k}=1)$ over all (ordered) $k$-tuples of distinct vertices  $(x_1,\ldots,x_k)$. \\
As $\Pd(I_{n,x_1}=\ldots=I_{n,x_k}=1)$ is the probability that no loop in $\mathcal{DL}_{\alpha_n}$ intersects the subset $F=\{x_1,\ldots,x_k\}$, 
\[
\Pd(I_{n,x_1}=\ldots=I_{n,x_k}=1)=\exp\Big(-\alpha_n\mu(\ell,\,  \ell\; \text{intersects}\; F)\Big)\] 
and 
\begin{multline*}\mu(\ell,\,  \ell\; \text{intersects}\; F)=\mu(\DL(X))-\mu(\DL(F^c))\\=\log\Big(\det(G)\prod_{x\in\enu{n}}\lambda_{x}\Big)-\log\Big(\det(G^{(F^c)})\prod_{x\in F^c}\lambda_{x}\Big).
\end{multline*}
 Therefore, ${\displaystyle \Ed(S_n)_k=k!\sum_{F\in \Psub{\enu{n}}{k}}\Big(\frac{\det(G^{(F^c)})}{\det(G)\prod_{x\in F}\lambda_x}\Big)^{\alpha_n}}$. \\
In our setting, for every $x\in\enu{n}$, $\lambda_x=\lambda^{(n)}=n(\epsilon+1)-1$  and  for every $D\subset \enu{n}$, ${\det(G^{(D)})^{-1}=(\lambda^{(n)}+1)^{|D|-1}(\lambda^{(n)}+1-|D|)}.$ Thus 
\begin{align*}\Ed(S_n)_k&=n(n-1)\ldots(n-k+1)\Big(1-\frac{1}{n(1+\epsilon)}\Big)^{-k\alpha_n}\Big(1+\frac{k}{n\epsilon}\Big)^{-\alpha_n}.\\  &=\prod_{i=1}^{k}(1-\frac{i}{n})\exp\Big(k(\log(n)-\frac{\alpha_n}{n\epsilon(1+\epsilon)})\\&\qquad\qquad\qquad\qquad\quad+k\frac{\alpha_n}{2n^2}(\frac{1}{(1+\epsilon)^2}+\frac{k}{\epsilon^{2}})+\alpha_n O(\frac{1}{n^3\epsilon^{3}})\Big).
\end{align*}
As $\alpha_n=n\epsilon(1+\epsilon)(\log(n)+a+o(1))$, we deduce that $\Ed(S_n)_k$ converges to $\exp(-ka)$ for every $k\in \NN^*$. 
The convergence to the Poisson distribution with parameter $e^{-a}$ follows from the theory of moments. 
\end{proof}

\begin{remark}\
\begin{itemize}
\item[(i)]The distribution of the number of isolated vertices at a time $\alpha$ can also be obtained by the inclusion-exclusion principle; the probability that there exists $r$ isolated vertices at time $\alpha$ is  \[\sum_{j=0}^{n-r}(-1)^{j}\frac{n!}{r!j!(n-r-j)!}\Big(1-\frac{1}{n(1+\epsilon)}\Big)^{-\alpha(r+j)}\Big(1+\frac{r+j}{n\epsilon}\Big)^{-\alpha}. \]
\item[(ii)]In the statement of Proposition \ref{nbisolated}, we assume that  the intensity of the killing measure $\kappa_n$ is proportional to the number of vertices $n$: $\kappa_n=n\epsilon$. 
In fact the same proof  shows that Proposition \ref{nbisolated} also holds if we only assume  that $\frac{\kappa_n}{\log(n)}$ converges to $+\infty$  and replace $\epsilon$ with $\epsilon_n=\frac{\kappa_n}{n}$ in the statement.  
 Let us note that if $\frac{\kappa_n}{\log(n)}$ converges to a constant $c>0$ then $\Ed(S_n)_k$ converges to $m_k=\exp(-ka+\frac{k^2}{2c})$ for every $k\in\NN^*$ and  $\frac{t^km_k}{k!}$ tends to $+\infty$ for any $t>0$. 
\end{itemize}
\end{remark}
\subsection{Asymptotics for coalescence time}
Let $\tau_n$ denote  the coalescence time, \textit{i.e.} the first time at which  $\mathcal{C}_{\alpha}^{(n)}$ has only one block. The following theorem shows that the cover time and the coalescent time have the same asymptotic distribution: 
\begin{theorem}
\label{coaltime}
 Let us assume that  $\kappa_n=n\epsilon$ for every $n\in\NN^*$ with $\epsilon>0$.  For every $n\in\NN^*$, set $\alpha_n=\epsilon(1+\epsilon)n(\log(n)+a+o(1))$ where $a$ is a fixed real. \\
For every $k\in \NN$, the probability that $\DL^{(n)}_{\alpha_n}$ consists only of a component of size greater or equal to 2 and $k$ isolated points converges to $\displaystyle{\exp(-e^{-a})\frac{e^{-ak}}{k!}}$ as $n$ tends to $+\infty$.\\
In particular, $\dfrac{\tau_n}{n\epsilon(1+\epsilon)}-\log(n)$ converges in law to the Gumbel distribution. 
\end{theorem}
\begin{remark}\
\begin{itemize}
\item[(i)] Theorem \ref{coaltime} shows that $\epsilon(\epsilon+1)n\log(n)$ is a sharp threshold function for the connectedness of the random graph process $(\mathcal{G}^{(n)}_{\alpha})_{\alpha\geq 0}$ defined by the loop sets $(\DL^{(n)}_\alpha)_{\alpha}$. 
\item[(ii)] 
Such properties have been proven for the first time by Erd\"os and R\'enyi in \cite{ErdosRenyi59} for the random graph model they have introduced. To facilitate comparison, the following theorem states some of their results in a slightly different way from that used in \cite{ErdosRenyi59}: 
\begin{bibli}[Erd\"os and R\'enyi, \cite{ErdosRenyi59}]
 Let $G(n,N)$ denote a random graph obtained by forming $N$ links between $n$ labelled vertices, each of the  $\displaystyle{\binom{N}{\binom{n}{2}}}$ graphs being equally likely.\\ 
For every $c\in\RR$ and every $k\in\NN$, the probability that  ${G(n,\lfloor\frac{n}{2}(\log(n)+c)\rfloor)}$ contains a connected component of size $n-k$ and $k$ isolated points converges to $\displaystyle{\exp(-e^{-c})\frac{e^{-ck}}{k!}}$ as $n$ tends to $+\infty$.
\end{bibli}
\item[(iii)]
Let us note that the ratio  $\dfrac{\mu(\ell,\ \ell\; \text{passes through exactly two vertices})}{\mu(\DL(\enu{n}))}$ 
 converges to $-\frac{1}{2}\Big(\epsilon+1+(\epsilon+1)^2\log(1-\frac{1}{\epsilon+1})\Big)^{-1}$   as  $n$ tends to $+\infty$. The limit is an increasing function of $\epsilon$ that converges to $1$ as $\epsilon$ tends to $+\infty$.   Therefore,  it is not surprising to obtain properties similar to  the Erd\"os-R\'enyi model for large $\epsilon$. \\
In fact, we can  show using the same proof that Theorem  \ref{coaltime} holds if we replace $\epsilon$ with a positive sequence $(\epsilon_n)_n$ such that  $\displaystyle{\liminf_{n\rightarrow +\infty}\log(n)\epsilon_n>0}$. \\
By contrast, if we let $\epsilon$ converge to $0$ as $n$ tends to $+\infty$, the loop sets can have very large loops: a result of Y.  Chang in  \cite{Chang12} implies that if $\epsilon\underset{n\rightarrow +\infty}{\sim} n^{-d}$ with $d>1$ then  
\[\Pd\big(\exists \ell \in \DL^{(n)}_{1/\log(n)}\;\text{covering}\;\enu{n}\big)\underset{n\rightarrow +\infty}{\rightarrow} 1-e^{-(d-1)}.\]
\end{itemize}
\end{remark}
\noindent The rest of the section is devoted to the proof of  Theorem \ref{coaltime}.\\
Let $A_{n,k}$ denote the event `\textit{$\DL^{(n)}_{\alpha_n}$ consists only of a component of size greater 
or equal to 2 and $k$ isolated points}', let $V_{n}$ be the number of isolated vertices in $\DL^{(n)}_{\alpha_n}$ 
and let $B_n$ be the event `\textit{$\DL^{(n)}_{\alpha_n}$ has at least two components of size greater or equal to $2$}'. 
For $n\geq k+2$, 
 \[\Pd(A_{n,k})=\Pd(V_n=k)-\Pd(\{V_n=k\}\cap B_n). \]
By Proposition \ref{nbisolated}, $\Pd(V_n=k)$ converges to $e^{-e^{-a}}\frac{e^{-ka}}{k!}$. 
We shall prove that $\Pd(B_n)$ converges to $0$. 
For a subset $F$, let $q_{F,n}$ denote the probability that $F$ is a block at time $\alpha_n$.  
\[ \Pd(B_n)\leq \sum_{r=2}^{\lfloor n/2\rfloor}\sum_{F\in \Psub{\enu{n}}{r}}q_{F,n}.\]
As $\DL^{(n,F)}_{\alpha_n}$ is independent of $\DL^{(n)}_{\alpha}\setminus\DL^{(n,F)}_{\alpha_n}$, we have 
\[
q_{F,n} =\Pd(\DL^{(n,F)}_{\alpha_n}\;\text{is connected})\Pd(\text{no  loop in}\;\DL^{(n)}_{\alpha_n}\; \text{intersects}\;F\:\text{and}\;F^c).
\]
For a sufficiently large set $F$, we shall simply  bound  $q_{F,n}$ from above by 
\[
\Pd(\text{no  loop in}\;\DL^{(n)}_{\alpha_n}\;\text{intersects}\;F\;\text{and}\;F^c).
\]  
For small $F$,  the probability $\Pd(\DL^{(n,F)}_{\alpha_n} \;\text{is connected})$ is small. Instead of computing it, we shall consider its upper bound  by the probability that the total length of non-trivial loops on $\mathcal{\DL}_{\alpha_n}^{(n,F)}$ is greater or equal to $|F|$, that is  $\Pd(\sum_{x\in F}N^{(\alpha_n,F)}_x\geq |F|)$ where  $N^{(\alpha,F)}_x$ denotes the number of crossings of a vertex $x$ by the set of non-trivial loops included in $F$ at time $\alpha$. Therefore for a small set $F$, we shall use that  \[q_{F,n} \leq \Pd(\sum_{x\in F}N^{(\alpha_n,F)}_x\geq |F|)\Pd(\text{no  loop in}\;\DL_{\alpha_n}\;\text{intersects}\;F\;\text{and}\;F^c).\]
We start by stating two lemmas: Lemma \ref{intersblock} provides an upper bound for the probability that 
no  loop in $\DL^{(n)}_{\alpha_n}$  intersects $F$  and $F^c$. Lemma \ref{lengthloop} gives  an  
exponential inequality  for $\Pd(\sum_{x\in F}N^{(\alpha_n,F)}_x\geq |F|)$ based on Markov's inequality. To shorten the notations let $\Psub{n}{}$ denote the power set of $\enu{n}$ and let $\Psub{n}{r}$ consist  of subsets of $\enu{n}$ of cardinality~$r$. 
\begin{lem}
\label{intersblock}
 For every nonempty proper subset $F$ of $\enu{n}$, let $A_{n,F}$ denote the event `no  loop in $\DL^{(n)}_{\alpha_n}$ passes through both $F$ and $F^c$' 
where ${\alpha_n=\epsilon(\epsilon+1)n(\log(n)+a+o(1))}$.\\ 
 For every  $\delta\in]0,1[$ there exists  $n_{\delta}>0$ such that for every  
 $n>n_{\delta}$ and for every  $n^{1-\delta}\leq r\leq \frac{n}{2}$, 
\[\sum_{F\in \Psub{n}{r}}\Pd(A_{n,F})\leq \frac{1}{\sqrt{r}}\exp\Big(-\frac{1-\delta}{2}n^{1-\delta} \log(n)\Big). \]
\end{lem}
\begin{proof}
Let $r\in\enu{n/2}$ and let $F$ be a subset of $\enu{n}$ with $r$ elements.
$$
\Pd(A_{n,F})=\exp\Big(-\alpha_n\mu(\{\ell,\ \ell \text{ intersects } F \text{ and } F^c\})\Big)=\Big(\frac{\det(G^{(F)})\det(G^{(F^c)})}{\det(G)}\Big)^{\alpha_n}.
$$ 
In our setting,  $\det(G^{(D)})=(n(1+\epsilon))^{-|D|+1}(n(1+\epsilon)-|D|)^{-1}$ for every subset $D$. Therefore,   $\Pd(A_{n,F})=(1+\frac{1}{\epsilon(\epsilon+1)}\frac{|F|}{n}(1-\frac{|F|}{n}))^{-\alpha_n}$. 
Using that $\binom{n}{r}\leq \frac{1}{\sqrt{2\pi r}\sqrt{1-\frac{r}{n}}}(\frac{n}{r})^r(1-\frac{r}{n})^{-(n-r)}$ (see for example \cite{BollobasBook}, formula 1.5 page 4),  we obtain:
\[\sum_{F\in \Psub{n}{r}}\Pd(A_{n,F}) \leq  \frac{1}{\sqrt{r}}\exp\big(-nf_n(\frac{r}{n})\big)\] 
where 
$f_n$ is the function on $]0, 1/2]$ defined by: \[f_n(x)=x\log(x)+(1-x)\log(1-x)+u_n\epsilon(1+\epsilon)\log\Big(1+ \frac{x(1-x)}{\epsilon(\epsilon+1)}\Big)\] for $x\in ]0, 1/2]$,
with $u_n=\log(n)+a+o(1)$. \\
The function $f_n$ is of class $C^2$ in the interval $]0,1/2]$ and the first two derivatives of $f_n$ are: 
\begin{itemize}
 \item $f'_{n}(x)=\log(x)-\log(1-x)+u_n\frac{1-2x}{1+\frac{x(1-x)}{\epsilon(\epsilon+1)}}$,
\item $f''_n(x) = \frac{Q_n(x(1-x))}{x(1-x)(1+\frac{x(1-x)}{\epsilon(\epsilon+1)})^2}$ where $Q_n$ is the polynomial function of second order defined, for every $y$, by: \[Q_n(y)=1+y\Big(\frac{2}{\epsilon(1+\epsilon)}-(2+\frac{1}{\epsilon(1+\epsilon)})u_n\Big)+y^2\frac{1}{\epsilon(1+\epsilon)}\Big(\frac{1}{\epsilon(1+\epsilon)}+2u_n\Big).\]
\end{itemize}
Thus  $f_n$ has the following properties: 
\begin{itemize}
 \item[(i)] $\lim_{x\rightarrow 0^{+}}f_n(x)=0$,  $f_n(1/2)>0$, 
\item[(ii)] $\lim_{x\rightarrow 0^{+}}f'(x)<0$,  $f'(1/2) = 0$, 
\item[(iii)] $f_n''=0$ has at most two solutions in $]0, 1/2]$.
\end{itemize}

A $C^2$ function in the interval $]0,1/2]$ that verifies conditions $(i)$, $(ii)$,  $(iii)$ has the following property:  for every $c\in]0, 1/2]$ such that $f_n(c)$ is positive, the minimum of $f_n$ on the interval $[c, 1/2]$ lies on $c$ or $\frac{1}{2}$ (see Lemma \ref{minfunct} in the appendix). In fact, $(f_n(1/2))_n$ converges to $+\infty$, whence for sufficiently large $n$, $\inf_{x\in[c, 1/2]}f_n(x)=f_n(c)$. 

To conclude, let us consider  $f_n$ at point $\frac{1}{n^\delta}$ for some $\delta\in]0, 1[$. 
\[f_{n}(\frac{1}{n^\delta})=-\frac{\delta\log(n)}{n^\delta}+(1-\frac{1}{n^\delta})\log(1-\frac{1}{n^\delta})
+u_n\epsilon(1+\epsilon)\log\Big(1+\frac{n^{-\delta}(1-n^{-\delta})}{\epsilon(1+\epsilon)}\Big).\] 
 Therefore $f_{n}(\frac{1}{n^\delta})$ is equivalent to $\frac{1-\delta}{n^\delta}\log(n)$. 
We deduce that for sufficiently large values of $n$ and  for every $x\in[\frac{1}{n^\delta},  1/2]$, $f_n(x)\geq f_n(\frac{1}{n^\delta})\geq \frac{1-\delta}{2n^{\delta}}\log(n)$ which ends the proof. 
\end{proof}
\begin{lem}
\label{lengthloop}
 Let $\delta$ and $\bar{\delta}$ be two positive reals such that  $0<\bar{\delta}<\delta<1$. 
Assume that $\kappa_n=n\epsilon$ with $\epsilon>0$ and set $\alpha_n=\epsilon(1+\epsilon)n(\log(n)+a+o(1))$ with $a\in \RR$ for every $n\in\NN$. \\
 There exists $n_{\delta,\bar{\delta}}>0$ such that for every $n\geq n_{\delta,\bar{\delta}}$,  and $F\in\Psub{n}{}$ with $2\leq |F|\leq n^{1-\delta}$, 
\[\Pd\Big(\sum_{x\in F}N^{(\alpha_n,F)}_x\geq |F|\Big)\leq n^{-\frac{\bar{\delta}}{2}|F|}.\]
\end{lem}
\begin{proof}
To prove Lemma \ref{lengthloop}, we shall apply  Markov's inequality to the random variable\linebreak ${\exp(\theta\sum_{x\in F}N^{(\alpha_n,F)}_x)}$ for a well-chosen positive real $\theta$. 

The generating function of the vector $(N^{(\alpha,F)}_x,x\in F)$ has been computed for any finite graph $\mathcal{G}=(\mathcal{V},\mathcal{E})$ in \cite{stfl}\footnote{Formula (4.3) page 37 in \cite{stfl} contains a misprint: the exponent $1$ in the left-hand side term of the equation has to be replaced by $\alpha$.}, page 37: $\forall\ (s_x)_{x\in \mathcal{V}}\in]0, 1]^{\mathcal{V}},$
\[\Ed\Big(\prod_{x\in F}s_x^{N^{(\alpha,F)}_{x}}\Big)=\left(\det\Big(\Big[s_x\delta_{x,y}+\lambda_x(1-s_x)G^{(F)}_{x,y}\Big]_{x,y\in F}\Big)\right)^{-\alpha}.\]
In our setting,  the generating function of $\sum_{x\in F}N^{(\alpha,F)}_x$ satisfies: \[\Ed\Big(s^{ \sum_{x\in F}N^{(\alpha,F)}_x}\Big)=\Big(\det(sI_{|F|}+\lambda^{(n)}(1-s)G^{(F)})\Big)^{-\alpha}\; \forall\, 0<s\leq 1 \]
where $\lambda^{(n)}=n(\epsilon+1)-1$ and $G^{(F)}_{x,y}=\frac{1}{\lambda^{(n)}+1}\Big(\un_{\{x=y\}}+\frac{1}{\lambda^{(n)}+1-|F|}\Big)$ for ${x,y\in F}$.  
The computation of that determinant is detailed in the Appendix (Lemma \ref{detmatrix}). We obtain:
 \begin{equation}
\label{detGF}
\det(sI_{|F|}+\lambda^{(n)}(1-s)G^{(F)})=\Big(1-\frac{1-s}{\lambda^{(n)}+1}\Big)^{|F|-1}\Big(1+(1-s)\frac{|F|-1}{\lambda^{(n)}+1-|F|}\Big).
\end{equation}
Equality \eqref{detGF} can be extended to $s\in \RR$ and the determinant  is positive for every $0\leq s<\frac{\lambda^{(n)}}{|F|-1}$.
By  Markov's inequality, for every $\theta>0$,  
\begin{equation}\label{ineqmarkovapp}
\Pd\Big(\sum_{x\in F}N^{(\alpha_n,F)}_x\geq |F|\Big)\leq  \exp(-\psi_{F,n}(\theta))
\end{equation}
 where $\psi_{F,n}(\theta)=-\log\Big(\Ed\big(\exp(-\theta |F| +\theta\sum_{x\in F}N^{(\alpha_n, F)}_x )\big)\Big)$. 
The expression of the generating function of $\sum_{x\in F}N^{(\alpha_n,F)}_x$ shows that the value of $\psi_{F,n}(\theta)$ depends on $F$ only via $|F|$; we  denote it $\psi_{|F|,n}(\theta)$. The value of $\psi_{r,n}(\theta)$ is
\[\psi_{r,n}(\theta)=\theta r +\alpha_n\Big((r-1)\log\big(1+\frac{e^{\theta}-1}{\lambda^{(n)}+1}\big)+\log\big(1-(e^\theta-1)\frac{r-1}{\lambda^{(n)}+1-r}\big)\Big) \]
for every $\theta>0$ such that $e^\theta\leq \frac{\lambda^{(n)}}{r-1}$.  \\
Set $\beta_{r,n}=\frac{\lambda^{(n)}+1}{\sqrt{\alpha_n(r-1)}}$ and $\theta_{r,n}=\log(\beta_{r,n})$ for every integer $2\leq r\leq n$. 
We shall use  inequality \eqref{ineqmarkovapp} for $\theta=\theta_{r,n}$. Before, let us note that $\min(\beta_{r,n},\,r\in\enum{2}{n^{1-\delta}})$ converges to  $+\infty$ and $\max(\frac{\beta_{r,n}(r-1)}{\lambda^{(n)}},\, {r\in\enum{2}{n^{1-\delta}}})$ converges to $0$. 
Therefore, for sufficiently large values of $n$ and $r\in\enum{2}{n^{1-\delta}}$, the two conditions $\theta_{r,n}>0$ and $\exp(\theta_{r,n})\leq \frac{\lambda^{(n)}}{r-1}$ are satisfied.  
\\
It remains to bound from below $\psi_{r,n}(\theta_{r,n})$.
\begin{multline*}\psi_{r,n}(\theta_{r,n})=r\Big(\log(\lambda^{(n)}+1)-\frac{1}{2}\log(\alpha_n)-\frac{1}{2}\log(r-1)\Big) \\ +\alpha_n\Big((r-1)\log(1+\frac{\beta_{r,n}-1}{\lambda^{(n)}+1})+\log(1-(r-1)\frac{\beta_{r,n}-1}{\lambda^{(n)}+1-r})\Big).
\end{multline*}
 Using the classical lower bounds  
\[\log(1-t)\geq -t-t^2\;\text{for}\;0<t<1/2\;\text{and}\;\log(1+t)>t-\frac{t^2}{2}\;\text{for}\; t>0\] 
along with $\max\Big((r-1)\frac{\beta_{r,n}}{\lambda^{(n)}+1},\, {r\in\enum{2}{n^{1-\delta}}}\Big)\underset{n\rightarrow+\infty}{\rightarrow} 0$, we obtain  \[\psi_{r,n}(\theta_{r,n})\geq  \frac{r}{2}(\log(n)-\log(r-1)-I_{n,1})-I_{n,2}-I_{n,3}\]
where, for sufficiently large values of $n$
\begin{itemize}
\item $I_{n,1}=\log(1+\frac{\epsilon}{1-\epsilon})+\log(\log(n)+a+o(1))\leq 2\log(\log(n))$;
\item $I_{n,2}=\alpha_n(r-1)r\frac{\beta_{r,n}-1}{(\lambda^{(n)}+1)(\lambda^{(n)}+1-r)}\leq 2r\sqrt{\frac{r-1}{n}(\log(n)+a+o(1))}$;
\item  $I_{n,3}=\alpha_n(r-1)(\beta_{r,n}-1)^2\Big(\frac{1}{2(\lambda^{(n)}+1)^2}+\frac{r-1}{(\lambda^{(n)}+1-r)^2}\Big)\leq \frac{1}{2}+2(r-1)$.
\end{itemize}
Therefore there exists a constant $M$ such that for sufficiently large values of $n$ and for every $2\leq r\leq \frac{n}{2}$, 
\[\psi_{r,n}(\theta_{r,n})\geq  \frac{r}{2}\big(\log(n)-\log(r)-2\log(\log(n))-M\big).\] 
In particular for $2\leq r\leq n^{1-\delta}$, $\psi_{r,n}(\theta_{r,n})\geq  \delta\frac{r}{2}\log(n)-2\log(\log(n))-M$. 
This shows that for every $0<\bar{\delta}<\delta$ and sufficiently large values of $n$, $\min_{r\in \enum{2}{n^{1-\delta}}}\psi_{r,n}(\theta_{r,n})\geq \frac{\bar{\delta}}{2}r\log(n)$ which ends the proof.  \\
Let us note that a study of $\psi_{r,n}$ shows that for every $r\in \enum{2}{n^{1-\delta}}$, and sufficiently large values of $n$, $\psi_{r,n}$ has a maximum at a point which is equivalent to $\theta_{r,n}$ as $n$ tends to $+\infty$. 
\end{proof}

We can  now complete the proof of Theorem \ref{coaltime}.  Set $S_{r,n}=\sum_{F\in \Psub{n}{r}}q_{F,n}$. We shall prove that the upper bound of $\Pd(B_n)$, $\sum_{r=2}^{\lfloor n/2\rfloor }S_{r,n}$ converges to $0$ as $n$ tends to $+\infty$. 
By Lemma \ref{intersblock},   
\[\sum_{r=n^{1-\delta}}^{\frac{n}{2}} S_{r,n}\leq n\exp\Big(-\frac{1-\delta}{2}n^{1-\delta}\log(n)\Big).\]
Let us consider the case of subsets $F$ with  $r\leq n^{1-\delta}$ elements.  Using the notations introduced in Lemmas \ref{intersblock} and \ref{lengthloop}, we have
\[S_{r,n}\leq \frac{1}{\sqrt{r}}\exp\Big(-nf_n(\frac{r}{n})-\psi_{r,n}(\theta_{r,n})\Big).\] 
By Lemma \ref{lengthloop}, for every $\bar{\delta}\in ]0, \delta[$, sufficiently large $n$  and   $r\in\enum{2}{n^{1-\delta}}$,  
\[S_{r,n}\leq \frac{1}{\sqrt{r}}\exp\Big(-n\big(f_{n}(\frac{r}{n})+\frac{\bar{\delta}}{2}\frac{r}{n}\log(n)\big)\Big).\] 
Let us study the function $\bar{f}_n(x)=f_n(x)+\frac{\bar{\delta}}{2}x\log(n)$. 
By computations we obtain that $\bar{f}_n(2/n)$ is equivalent to $\frac{\bar{\delta}}{2n}\log(n)$ as $n$ tends to $+\infty$. The study of $f_n$ made in the proof of Lemma  \ref{intersblock} shows that, for sufficiently large values of $n$, $\bar{f}_n$ is greater than  $\bar{f}_n(2/n)$ in   $[2/n, 1/2]$. Let us introduce  a real $\bar{\bar{\delta}}$ such that $0<\bar{\bar{\delta}}<\bar{\delta}$. We have shown that for $n$ large enough and every $r\in\enum{2}{n^{1-\delta}}$,  $S_{r,n}\leq \frac{1}{\sqrt{r}}n^{-\frac{\bar{\bar{\delta}}}{2}}$ and thus 
\[\sum_{r=2}^{n^{1-\delta}}S_{r,n}\leq n^{1-\delta-\bar{\bar{\delta}}/2}.\]
By taking $\delta=3/4$ and $\bar{\bar{\delta}}=2/3$ for instance we obtain that for sufficiently large values of $n$, $P(B_n)\leq ne^{-\frac{1}{8}n^{1/4}\log(n)}+n^{-1/12}$. This ends the proof of Theorem \ref{coaltime}.
\appendix
\section{}
\begin{lem}
\label{permanent}
For a finite set $E$, let $\mathfrak{S}_E$  denote the set of permutations of $E$ and let  $\mathfrak{S}^{0}_{E}$ consist of permutations of $E$ without
fixed point. For two integers $r\geq 2$ and $1\leq k\leq r/2$, let  $\mathfrak{P}_{k}(2,r)$ denote the partitions of $\enu{r}$ with $k$ blocks, each of them having at least two elements. \\ 
 For a $r\times r$ matrix $A$ and a real $\alpha$, set 
\[\Per_{\alpha}^{0}(A)=\sum_{\sigma
\in\mathfrak{S}_{\enu{r}}^{0}}\alpha^{m(\sigma)}A_{1,\sigma(1)}\cdots A_{r,\sigma(r)}\] where  $m(\sigma)$ denotes the number of cycles in a permutation $\sigma$ ($\Per_{\alpha}^{0}(A)$ introduced in \cite{stfl} page 41, can also be defined as the $\alpha$-permanent of $A$ with the diagonal elements  of the matrix set to zero).\\
Another expression of $\Per_{r}^{0}(A)$ is 
\[\sum_{k=1}^{r/2}\alpha^{k}\sum_{(\pi_1,\ldots,\pi_k)\in \mathfrak{P}_k(2,r)}\prod_{j=1}^{k}\Big(\frac{1}{|\pi_j|}\sum_{\sigma\in \mathfrak{S}_{\pi_j}}A_{\sigma(1),\sigma(2)}A_{\sigma(2),\sigma(3)}\cdots A_{\sigma(|\pi_j|),\sigma(1)}\Big).\]
\end{lem}
\begin{proof}
 For a finite set $B$, let $\mathfrak{S}^{c}_{B}$ be the subset of $\mathfrak{S}_{B}$ which consists of cycles of length $|B|$.
The decomposition of permutations into disjoint cycles  entails that
\[\Per_{r}^{0}(A)=\sum_{k=1}^{r/2}\alpha^{k}\sum_{\pi=(\pi_1,\ldots,\pi_k)\in \mathfrak{P}_k(2,r)}\prod_{j=1}^{k}\Big(\sum_{\sigma_j\in \mathfrak{S}^{c}_{\pi_j}}\prod_{u\in\pi_j}A_{u,\sigma_j(u)}\Big).\]
Therefore, it remains to prove that for every $k\geq 2$
\begin{equation}
\label{eqcycle}
\sum_{\sigma\in \mathfrak{S}^{c}_{\enu{k}}}\prod_{u=1}^{k}A_{u,\sigma(u)}=\frac{1}{k}\sum_{\nu\in \mathfrak{S}_{\enu{k}}}A_{\nu(1),\nu(2)}A_{\nu(2),\nu(3)}\ldots A_{\nu(k),\nu(1)}. 
\end{equation}
Starting from a permutation $\nu\in\mathfrak{S}_{\enu{k}}$, we define a cycle of length $k$, $F(\nu)$ such that  \[\prod_{i=1}^{k}A_{i,F(\nu)(i)}=A_{\nu(1),\nu(2)}A_{\nu(2),\nu(3)}\ldots A_{\nu(k),\nu(1)}
\] by setting
\[
F(\nu)(u)=\left\{\begin{array}{ll}\nu(1)&\;\text{if}\;\nu^{-1}(u)=k\\
\nu(\nu^{-1}(u)+1)&\;\text{if}\;\nu^{-1}(u)\leq k-1.
\end{array}\right.
\]
Conversely, starting from a cycle $\sigma$ of length $k$, we can construct exactly $k$ different permutations 
$\nu_1$, \ldots, $\nu_k$ such that $F(\nu_1)=\ldots =F(\nu_k)=\sigma$ by setting  
 \[\nu_i(j)=\left\{\begin{array}{ll}i &\;\text{if}\; j=1\\
\sigma^{j-1}(i)   &\;\text{if}\; j\in\{2,\ldots,k\}
  \end{array}\right.\]
This shows equality \eqref{eqcycle}, completing the proof. 
\end{proof}
\begin{lem} \label{minfunct}
Let $a<b$ be two reals. 
Let $f$ be a real function of class $C^2$ in the interval $]a, b]$. Assume that:
\begin{itemize}
 \item $\lim_{x\rightarrow a^{+}}f(x)=0$,  $f(b)>0$, 
\item $\lim_{x\rightarrow a^{+}}f'(x)<0$,  $f'(b)\leq 0$, 
\item $f''=0$ has at most two solutions in $]a, b]$.
\end{itemize}
Then for every $c\in]a, b]$ such that $f(c)>0$, the minimum of $f$ on the interval $[c, b]$ lies on $c$ or~$b$. 
\end{lem}
 \begin{proof}
The assumptions on $f$ ensure that the only possible configurations are:
\begin{itemize}
 \item[(i)] there exists $\rho\in]a, b[$ such that $f$ is a decreasing function on $]a, \rho[$ and an increasing function on $[\rho, b]$;
\item[(ii)] there exists two reals $\rho_{-}< \rho_{+}$ in $]a, b[$ such that $f$ is a decreasing function on $]a, \rho_{-}[$, an increasing function on $[\rho_{-}, \rho_{+}]$ and a decreasing function on $[\rho_{+}, b]$. 
\end{itemize}
Since $\lim_{x\rightarrow a^{+}}f(x)=0$, for every $x\in]a, b[$ such that $f(x)>0$, we have in both cases  ${\underset{u\in[x, b]}{\inf}f(u)=\min(f(x),f(b))}$. 
\end{proof}
\begin{lem}\label{detmatrix}
 Let $a$ and $b$ be two reals and let $n$ be a positive integer. Let $I_n$ denote the identity matrix of size $n$ and let $J_n$ denote the $n\times n$  matrix with all entries equal to one.  The determinant of the matrix  $M_n=bJ_{n}+(a-b)I_{n}$ is $(a-b)^{n-1}(a+(n-1)b)$.  
\end{lem}
\begin{proof}
First, $\det(M_1)=a$. Let $n\geq 2$. 
 If we substract the column $n-1$ to the column $n$ and expand the determinant along the  column $n$, we obtain: \\ $\det(M_n)=(a-b)\det(M_{n-1})+(a-b)\det(R_{n-1})$ where $R_{n}$ denotes the matrix defined by \[R_{n}=\left\{\begin{array}{ll}b &\quad\text{if}\quad n=1   \\                    
b J_n +\diag(a-b,\ldots,a-b,0)&\quad\text{if}\quad n\geq 2.\\
\end{array}
\right.\] 
 By  applying the same transformations to $R_n$, we obtain 
\[\det(R_n)=(a-b)\det(R_{n-1})=(a-b)^{n-1}b.\] The expression of $\det(M_n)$ follows by induction on $n$.  
\end{proof}

\end{document}